\documentclass[singlespace]{article}

\usepackage{cite}

\usepackage[utf8]{inputenc} 
\usepackage[T1]{fontenc}    
\usepackage{hyperref}       
\usepackage{url}            
\usepackage{booktabs}       
\usepackage{amsfonts}       
\usepackage{nicefrac}       
\usepackage{microtype}      
\usepackage{amsmath,amsfonts,amsthm,amssymb}
\usepackage{mathrsfs}
\usepackage{algorithmic}
\usepackage[ruled,vlined]{algorithm2e}
\usepackage{verbatim}
\usepackage{cleveref}
\usepackage{color}
\usepackage{xcolor}
\usepackage{listings}
\usepackage{epsfig}
\usepackage{footnote}
\usepackage{subfigure}
\usepackage{graphicx}
\usepackage{siunitx}

\usepackage[american]{babel}
\usepackage[nodayofweek,level]{datetime}
\newcommand{\mydate}{\formatdate{19}{6}{2020}}
\date{\mydate}

\allowdisplaybreaks

\newtheorem{theorem}{Theorem}[section]
\newtheorem{lemma}[theorem]{Lemma}

\newtheorem{proposition}[theorem]{Proposition}
\newtheorem{definition}[theorem]{Definition}

\newtheorem{assumption}[theorem]{Assumption}

\newif\ifsubmit

  \submitfalse

\ifsubmit
\newcommand{\peijun}[1]{}
\newcommand{\ruoyu}[1]{}
\newcommand{\tianye}[1]{}
\else
\newcommand{\peijun}[1]{{\color{blue}{[Peijun: #1]}}}
\newcommand{\ruoyu}[1]{{\color{red}{[Ruoyu: #1]}}}
\newcommand{\tianye}[1]{{\color{orange}{[Tian: #1]}}}
\fi

\title{DEED: A General Quantization Scheme for Communication Efficiency in Bits}

\author{
Tian Ye
\thanks{Institute for Interdisciplinary Information Sciences, Tsinghua University, Beijing, Beijing
  (\texttt{yet17@mails.tsinghua.edu.cn}).
  The work is performed when Tian Ye is visiting
 UIUC.}
\and
Peijun Xiao
\thanks{Coordinated Science Laboratory, Department of ISE, University of Illinois at Urbana-Champaign, Urbana, IL 
  (\texttt{peijunx2@illinois.edu}).}
\and 
Ruoyu Sun
\thanks{Coordinated Science Laboratory, Department of ISE, University of Illinois at Urbana-Champaign, Urbana, IL (\texttt{ruoyus@illinois.edu}).}
}

\begin{document}

\maketitle

\begin{abstract}
In distributed optimization, a popular technique to reduce communication is quantization. In this paper, we provide a general analysis framework for inexact gradient descent that is applicable to quantization schemes.
We also propose a quantization scheme Double Encoding and Error Diminishing (DEED).
DEED can achieve small communication complexity in three settings: 
frequent-communication large-memory, frequent-communication small-memory, and 
infrequent-communication (e.g. federated learning). 
More specifically, in the frequent-communication large-memory setting, DEED can be easily combined with Nesterov's method, so that the total number of bits required is $ \tilde{O}( \sqrt{\kappa} \log 1/\epsilon )$, where $\tilde{O}$ hides numerical constant and $\log \kappa $ factors. 
In the frequent-communication small-memory setting, DEED combined with SGD only requires $\tilde{O}( \kappa  \log 1/\epsilon)$ number of bits in the interpolation regime. 
In the infrequent communication setting, DEED combined with Federated averaging requires a smaller total number of bits than Federated Averaging. 
All these algorithms converge at the same rate as their non-quantized versions, 
while using a smaller number of bits.
\end{abstract}



\section{Introduction}\label{sec: introduction}

There is a surge of interest in distributed learning for large-scale computation in recent decade \cite{recht2011hogwild, dean2012large,  coates2013deep, li2014scaling, chilimbi2014project, xing2015petuum, chen2015mxnet, abadi2016tensorflow}. 
In the past few years, new application scenarios such as
multi-GPU computation \cite{lin2017deep, seide20141, gupta2015deep, alistarh2017qsgd, wen2017terngrad},  mobile edge computing \cite{mcmahan2016communication, mao2017survey} and federated learning \cite{li2019federated, yang2019federated}
have received much attention.
These systems are often bandwidth limited,
 and one important question in distributed learning is how to reduce the communication complexity.

 


A natural method to reduce the communication complexity 
is to compress the gradients transmitted between the machines.  A host of recent works proposed to quantize the gradients and transfer the quantized gradients to save the communication cost \cite{lin2017deep, seide20141, gupta2015deep, alistarh2017qsgd, wen2017terngrad}. These gradient quantization methods are successfully applied on training large-scale problems, and are shown to achieve similar performance to the original methods using less training time \cite{alistarh2017qsgd, wen2017terngrad}.
Nevertheless, their theoretical properties, especially the relation with
their un-quantized versions, are not well understood. 
A recent work \cite{mishchenko2019distributed} proposed DIANA and proved its linear convergence rate in a large-memory setting.
Another work \cite{liu2019double} proposed DORE which converges linearly
to a bounded region in a small-memory setting. 
Nevertheless, these works often study a specific distribution optimization problem,
and do not directly apply to other settings.
For instance, they assume frequent communication, thus do not immediately apply to infrequent-communication due to the additional error caused by the infrequent updates. 
While it may be possible to further adapt these works to new settings,
one might still wonder whether there is a unified and principled 
method to design quantization methods.

Our work is motivated by the classical work on inexact gradient descent (GD) \cite{bertsekas2000gradient}, which provides a unified theorem
for inexact GD and covers a number of algorithms (including SGD). 
We provide a general analysis framework
for inexact GD that is tailored for quantization schemes.
We also propose a quantization scheme Double Encoding and Error Diminishing (DEED).
We summary our contributions below.
\begin{itemize}
    \item \textbf{General convergence analysis}.
    We provide a general convergence analysis for inexact gradient descent algorithms using absolute errors in encoding.
This can potentially cover a large number of quantized gradient-type methods. 
    \item \textbf{General quantization scheme}. We propose a general quantization scheme Double Encoding and Error Diminishing (DEED).
    This scheme can be easily combined with existing
     optimization methods, and provably saves bits in communication
     in three common settings of distributed optimization.
     
     \item \textbf{Improved communication complexity}.
     In the most basic setting of large-memory frequent-communication, 
    our theoretical bound of DEED-GD (apply DEED to GD) 
     is at least $ F $ times better than existing works,
     where $ F  $ is the number of bits representing a real number. 
    
\end{itemize}
The motivation and details of our general convergence analysis and general quantization scheme will be given in Section \ref{sec: general framework} and Section \ref{sec: general scheme}. Discussions of related work are in the appendix. We now summarize the results of our general quantization scheme in three common settings.

\textbf{Frequent-communication large-memory}. We propose an algorithm DEED-GD and show it converges linearly while saves communication in bits. We further combine DEED-GD with Nesterov's momentum to obtain an accelerated version called A-DEED-GD, which achieves the state-of-art convergence rate and save the most total number of bits in communication as shown in Table \ref{table: GD}.

\textbf{Frequent-communication small-memory}. We adopt the Weak Growth Condition from \cite{vaswani2018fast} and prove that our algorithm DEED-SGD converges to the optimal solution at a linear rate. We compute the total number of bits to achieve a certain accuracy for both our algorithm and other works. The comparison is presented in Table \ref{table: SGD}.

\textbf{Infrequent-communication.} We propose DEED-Fed and provide the first explicit bound on the number of bits to achieve a certain accuracy in Federate Learning under non-i.i.d assumptions.
Our results can be applied to both large-memory and small memory settings
and both full-participant and partial-participant settings.
To our best knowledge, \cite{reisizadeh2019fedpaq} is the only work that also provide a convergence rate for infrequent-communication setting under realistic assumptions. However, due to the limitation of their framework, they did not do quantization in broadcasting step, which results in a great waste of communication. Our algorithm could save up to $FdN$ bits per $E$ iterations\footnote{Please check Table \ref{table_large} for definition of $F, d$ and $N$.}.

\begin{savenotes}
\begin{table}[h]\label{table_large}
    \centering
    \renewcommand\arraystretch{1.26}
    \begin{tabular}{|c|c|c|c|}
    \hline
        Algorithm & Iterations & Bits per iteration & Total bits \\
    \hline
        \textbf{DEED-GD} & $\tilde{O}(\kappa\log\frac{1}{\varepsilon})$ & $\tilde{O}(dN)$ & $\tilde{O}(dN\kappa\log\frac{1}{\varepsilon})$\\
    \hline
        \textbf{A-DEED-GD} & $\tilde{O}(\sqrt{\kappa}\log\frac{1}{\varepsilon})$ & $\tilde{O}(dN)$ & $\tilde{O}(dN\sqrt{\kappa}\log\frac{1}{\varepsilon})$\\
    \hline
        \textbf{DIANA}\cite{mishchenko2019distributed} & $\tilde{O}(\kappa\log\frac{1}{\varepsilon})$ & $\tilde{O}(dNC)$ & $\tilde{O}(dNC\kappa\log\frac{1}{\varepsilon})$\\
    \hline
        \textbf{ADIANA}
        \cite{li2020acceleration} & $\tilde{O}(\sqrt{\kappa}\log\frac{1}{\varepsilon})$ & $\tilde{O}(dNC)$ & $\tilde{O}(dNC\sqrt{\kappa}\log\frac{1}{\varepsilon})$\\
    \hline
        \textbf{DORE}\cite{liu2019double} & $\tilde{O}(\kappa\log\frac{1}{\varepsilon})$ & $\tilde{O}(dN)$ & $\tilde{O}(dN\kappa\log\frac{1}{\varepsilon})$\\ 
    \hline
        \textbf{DQGD}\cite{khirirat2018distributed} & / & $\tilde{O}(dNC)$ & / \\ 
    \hline
        \textbf{QSVRG}\cite{alistarh2017qsgd} & $\tilde{O}(\kappa\log\frac{1}{\varepsilon})$ & $\tilde{O}(dNC)$ & $\tilde{O}(dNC\kappa\log\frac{1}{\varepsilon})$\\
    \hline
    \end{tabular}
    \vspace{1.2mm}
    \caption{Summary of our theoretical results in minimizing a strongly-convex function as \eqref{main problem} in large-memory setting with $N$ computing nodes. We denote the condition number of the problem as $\kappa$ and the problem dimension as $d$. $\tilde{O}(\cdot)$ omits $\log\kappa$ and constant terms. For algorithm \textbf{DIANA}, \textbf{ADIANA}, \textbf{DQGD} and \textbf{QSVRG}, because they didn't do double quantization, they can choose either broadcasting (fully connected network) or transmitting full gradient from center to workers (star network). $C=N$ in the first case, and $C=F$ otherwise, where $F$ is the number of bits representing a real number. 
    \textbf{DQGD} cannot converge to the optimal solution, and is not
    directly comparable to our result; so we use / in the iteration cell. 
    }
    \label{table: GD}
\end{table}
\end{savenotes}

\begin{table}[t]\label{table_small}
    \centering
    \renewcommand\arraystretch{1.26}
    \begin{tabular}{|c|c|c|c|}
    \hline
        Algorithm & Iterations & Bits per iteration & Total bits \\
    \hline
        \textbf{DEED-SGD} (WGC) & $\tilde{O}(\kappa\log\frac{1}{\varepsilon})$ & $\tilde{O}(dN)$ & $\tilde{O}(dN\kappa\log\frac{1}{\varepsilon})$\\
    \hline
        \textbf{DIANA}\cite{mishchenko2019distributed} & $\tilde{O}(\frac{\kappa}{\varepsilon^2})$ & $\tilde{O}(dNC)$ & $\tilde{O}(dNC\frac{\kappa}{\varepsilon^2})$\\
    \hline
        \textbf{DORE}\cite{liu2019double} & / & $\tilde{O}(dN)$ & / \\
    \hline
        \textbf{QSVRG}\cite{alistarh2017qsgd} & $\tilde{O}(\kappa\log\frac{1}{\varepsilon})$ & $\tilde{O}(dNC)$ & $\tilde{O}(dNC\kappa\log\frac{1}{\varepsilon})$\\
    \hline
    \end{tabular}
    \vspace{1.2mm}
    \caption{Summary of our theoretical results in solving problem \ref{main problem} in small-memory setting. All notations are the same as table \ref{table_large}.}
    \label{table: SGD}
    \vspace{-3mm}
\end{table}

\section{General Analysis Framework}\label{sec: general framework}

\subsection{Motivation}

Our work is inspired by the classical work by Bertsekas and Tsitsiklis \cite{bertsekas2000gradient} 
which  provides a general convergence analysis of inexact gradient descent methods. 
Since quantization also introduces error to the update direction, 
a natural idea is to apply the general framework of \cite{bertsekas2000gradient}
to design and analyze quantization methods. 
However, directly applying \cite{bertsekas2000gradient} may not provide the best result, because in quantization methods, we have a rather strong control of the ``error'' in the algorithm.
This situation is different from the worst-case or random error considered in \cite{bertsekas2000gradient}. 
Our idea is to develop a modified analysis framework that can accomodate the extra freedom of controlling the quantization error,
so as to obtain stronger results compared to directly applying \cite{bertsekas2000gradient}.
We hope that such a general frame work can help us design and analyze quantization algorithms for different settings in a unified manner. 

 
A key element in this analysis is to use absolute error in quantization. Many theoretical works on quantization methods do not
explicitly consider absolute error, but focus on relative error\footnote{The definition of quantization by absolute error $\alpha$ is in Definition \ref{def: absolute error}. The definition of quantization by relative error $\alpha$ is similar. We just need to replace $\alpha$ in the RHS of the inequality in Definition \ref{def: absolute error} by $\frac{\alpha}{\|w\|}$.} \cite{alistarh2017qsgd, khirirat2018distributed, tang2019doublesqueeze}. These two types of errors are equivalent in one step quantization and only differ by scaling,
but not equivalent in a multi-iteration convergence analysis.
In our framework, we use absolute error in quantization so that 
we can sum up the absolute error over iterates and control the rate.


\textbf{Overview of our general analysis framework.} We first discuss why we use absolute error in quantization, and then we provide a general analysis of convergence rate. Our general analysis is not limited to algorithms that perform quantization on gradient or gradient difference. Algorithms that do quantization on weights or combination of weights and gradients  are also covered in our framework.
In this general analysis framework, the key component is ``effective error'' which we define as the error occurred at the weight. For example, in frequent-communication large-memory setting, the effective error is absolute error times the learning rate.
We show that an algorithm converges only if the effective error diminishes to zero.
In addition, we establish the convergence rate in terms of the learning rate
and the effective error. 
As promised, this is a general framework, so it should be applicable to new settings with similar proofs as we will show in Section \ref{subsec: theoratical analysis}.

\subsection{Content of the Framework}

We consider a star-network where there are $N$ computing
nodes and one central node.
Suppose we want to minimize a function $f: \mathbb{R}^d \rightarrow \mathbb{R}$ decomposed as 
\begin{align}
\label{main problem}
    f(w)=\frac{1}{N}\sum\limits_{i=1}^N f_i(w),
\end{align}
where each function $f_i$ is held on the $i$-th machine (or computing node), $i=1, \dots, N$. We assume each $f_i$ is $L$-smooth and $f$ is $\mu$-strongly convex. 
We define the condition number $\kappa:=\frac{L}{\mu}$. The formal definition of $L$-smoothness and strong convexity are standard, so are given in the appendix.

\begin{definition}\label{def: absolute error}
\textbf{Absolute error encoding-decoding procedure}. 
An $\alpha$-encoding-scheme of a vector $w$ consists of an encoding algorithm $E:S\times\Xi\rightarrow\mathbb{Z}^+$ and a decoding algorithm $D:\mathbb{Z}^+\rightarrow\mathbb{R}^d$, where $S$ is the set of vector we need to quantize, $\Xi$ is the set of random seeds, $\mathbb{Z}^+$ is the set of positive integers, and $\mathbb{R}$ stands for the real domain. We assume:
\begin{itemize}
    \item Unbiased coding, i.e. $\mathbb{E}_\xi[D\circ E(w,\xi)]=w$.
    \item The absolute error is bounded by $\omega$, i.e
    .  $\|D\circ E(w,\xi)-w\|\leq \alpha. $
\end{itemize}
Besides, the number of bits of this procedure is $\mathbb{E}_\xi\lceil\log_2\left| E(S,\xi)\right|\rceil$.

\end{definition}


The lemma below gives an upper bound and lower bound number of 
the bits with given precision.

\begin{lemma}\label{quan_vect_lemma}
Given a set $S=\left\{\left.x\in\mathbb{R}^d\right| \|w\|_2\leq M\right\}$, any (random) quantization algorithm that encoding a vector in $S$ by absolute error $\sigma$
takes at least $\left\lceil d\log_2 \frac{1}{\varepsilon}\right\rceil$ (in expectation) number of bits, where $\epsilon = \frac{\sigma}{M}$ . In addition, there exists a (random) algorithm that takes only $\left\lceil 1.05d + d\cdot\log_2 \frac{1+2\varepsilon}{\varepsilon}\right\rceil$ bits (in expectation)\footnote{This bound is pessimistic when $\varepsilon$ is large. For example, when $\varepsilon\geq1$, the lower bound is 0, however, we cannot only use 0 bits because of unbiased property. In this case, we can use sparsity to get lower bits.}.
\end{lemma}

For convenience, we define $Q(\cdot, \varepsilon)$ as a coding procedure with maximal precision $\varepsilon$ with corresponding encoding and decoding procedure $E_\varepsilon$ and $D_\varepsilon$. The output vector $Q(w,\varepsilon)$ is $D_\varepsilon\circ E_\varepsilon(w,\xi)$. 


To derive a general analysis for quantized GD in minimizing the problem \eqref{main problem}, we consider a general series of functions $F_t: \mathbb{R}^d\rightarrow\mathbb{R}^d$, $t\geq0. $
The definition of $F$ depends on the specific problems.
For example, in frequent-communication, $F_t(w)$ is defined as the function mapping  $w_t$ to $w_{t+1}$ in the $t$-th iteration. 
The only assumption on $F_t$ in our framework is $F_t$ is a continuous function with Lipschitz constant $c_t<1$ with the same fixed point $w^*$. In most cases, this can be easily derived by strongly convexity assumption.



\begin{assumption}\label{assumption: Ft}
For $t = 1, 2, \cdots$, $F_t$ is a continuous function with Lipschitz constant $c_t<1$ and denote $w^*$ as the unique fixed point of all $F_t$.
\end{assumption}

\begin{theorem}\label{thm: general convergence}
Suppose Assumption \ref{assumption: Ft} holds and $\{w_t\}$ is a sequence generated by \begin{align}\label{eqn: general iterate update }
    w_{t+1}=F_t(w_t)+e_t,
\end{align}
for some chosen initial value $w_0$ and $e_t$ is a zero-mean random  noise depending on the (iteration) history and is bounded by $\alpha_t$. Define series $C_k^2=\sum\limits_{i=0}^{k-1}\alpha_i^2\prod\limits_{j=i+1}^{k-1}c_j^2$ and $D_k^2=\prod\limits_{i=0}^{k-1}c_i^2$.
Then we have 
\begin{eqnarray}\label{inequ_general_convergence}
\mathbb{E}\left[\|w_{T}-w^*\|^2\right]\leq D_T^2\|w_0-w^*\|^2+C_T^2.
\end{eqnarray}
In addition, there exists functions series $\{F_i\}_{i\geq0}$ and 
 noise $ \{ e_t \}_{t\geq 0} $ to make the inequality hold. Besides, if we suppose the sequence of the Lipschitz constants $\{c_i\}$ is non-decreasing, then the right hand side of (\ref{inequ_general_convergence}) converges linearly if and only if all $c_t$'s are always bounded above by a constant $c<1$ and $\alpha_t$ converges to 0 linearly.
\end{theorem}

\textbf{Remark.} We leave the deterministic version of Theorem \ref{thm: general convergence} in the appendix. It can be useful in proving the convergence of deterministic algorithm.

According to Theorem \ref{thm: general convergence}, to make $\{w_t\}$ converge to $w^*$, we need both $D_k$ and $C_k$ converge to 0. In frequent communication setting, the $D_k\rightarrow0$ implies the summation of learning rate diverges.
Then $C_k\rightarrow0$ implies the effective error converges to 0.

The last statement of theorem \ref{thm: general convergence} implies that for any quantized GD algorithms under our framework, we should take constant learning rate and linearly decreasing absolute error for linear convergence.


\section{Application of DEED in Three Settings} \label{sec: general scheme}

Based on Theorem \ref{thm: general convergence}, we notice that using diminishing error in each iteration can guarantee fast convergence. However, according to lemma \ref{quan_vect_lemma}, the maximal norm of the vector we want to quantize should also be diminishing, otherwise the number of bits may explode. To avoid explosion, we choose to quantize on gradient difference instead of gradient. The intuition is that $\|\nabla f_i(w_{t+1})-\nabla f_i(w_t)\|\leq L\|w_{t+1}-w_t\|$ 
who goes to zero as the iterate sequence converges. Finally, to save the communication in broadcasting, we perform quantization both on the computation nodes and the center node, i.e. ``double encoding''. We name our general quantization scheme as Double Encoding and Error Diminishing (DEED).

Based on the general quantization scheme DEED, we introduce algorithms for three common settings in distributed optimization for Problem \eqref{main problem}:     
frequent-communication large-memory, 
frequent-communication small-memory, and 
infrequent-communication. 
Frequent-communication means the every computing node communicates with the center node after every update, while this is not the case in infrequent-communication. 
In large-memory setting, each local server $f_i$ has enough memory to hold its data and use them to compute the full-batch gradient of $f_i$. 
In limited memory setting (e.g. only one GPU is available in computing) that each server is only able to compute the stochastic gradients since the data cannot be fit into one server.

\subsection{Frequent-communication large-memory setting}
We distinguish the large-memory setting and small-memory setting
for the following reasons.

First, from a practical side, different system designers have different memory budget. Some big companies can perform computation using 10,000+ GPUs or CPUs (e.g. \cite{goyal2017accurate,you2018imagenet,yamazaki2019yet}),
 while most researchers and companies can only use few GPUs or 
 a moderate number of CPUs. The problems they are facing are indeed different, since in large-memory setting we can implement full-batch GD
 \footnote{Disclaimer: we discuss large-memory setting
mainly due to theoretical interest. We do not run simulation on
 a large number of GPUs, though we will mimic 
 the large-memory setting. }
 (or large-batch SGD which are quite close to full-batch GD).
Note that ``large'' is a relative term;
 if the system designer has only 10 CPUs or even 2 GPUs, but
 all data can be loaded into the memory of these machines,
 then this is also a large-memory setting. 
 In a small-memory setting, 
 we can only load a mini-batch of the dataset into one machine at a time.
This necessitates the usage of Stochastic Gradient Descent (SGD).

 
Second, from a theoretical side, quantized gradient methods should be no better than gradient methods that utilize infinite-bandwidth. 
To judge the performance of quantized gradient methods,
one useful metric is the gap between quantized methods and
their non-quantized counter-parts.
It is impossible to prove linear convergence of
quantized SGD in the limited-memory setting without further assumptions, since even with infinite bandwidth SGD cannot achieve linear convergence rate \cite{luo1991convergence}. In contrast, with infinite bandwidth GD can achieve linear convergence rate. Due to different upper limits, large-memory and small-memory settings should be treated separately.

The frequent-communication large-memory version of DEED is described in Algorithm \ref{our_gd}.



\begin{algorithm}[ht]
\SetAlgoLined
 Initialization: Each server $i\in[N]$ holds $w_0=s_{-1}^i=v_{-1}=0$, server 0 holds $v_{-1}=0$, $k=0$\;
Hyper-parameters:  $\eta\in\left(0, \frac{2}{L+\mu}\right]$, $c=1-\eta\mu$, $ c'\in (c, 1) $; parameter $s \in \mathbb{R}_+ $\;
 \While{the precision is not enough}{
  \For{$i\in[N]$}{
   server $i$ computes $g_k^i=\nabla f_i(w_k)$\;
   server $i$ does quantization $d_k^i=Q(g_k^i - s_{k-1}^i, \frac{s{c'}^{k+1}}{2})$\;
   server $i$ updates $s_k^i = d_k^i + s_{k-1}^i$\;
   server $i$ send $d_k^i$ to server 0\;
  }
  server 0 computes $s_k = \frac{1}{N}\sum\limits_{i=1}^N d_k^i + s_{k-1}$\;
  server 0 does quantization $u_k = Q(s_k - v_{k-1}, \frac{s{c'}^{k+1}}{2})$\;
  server 0 sends $u_k$ to server $i,\forall i\in[N]$\;
  server 0 updates $v_k = u_k + v_{k-1}$\;
  \For{$i\in[N]$}{
   server $i$ updates $v_k = u_k + v_{k-1}$\;
   server $i$ updates $w_{k+1}= w_k - \eta v_k$\;
  }
 $k = k + 1$\;
 }
 \caption{Double Encoding and Error Diminishing Gradient Descent (DEED-GD)}
 \label{our_gd}
\end{algorithm}

\subsection{Frequent-communication small-memory}\label{sec: Quantization in stochastic gradient descent}


Now we consider the small-memory setting with frequent-communication. 
As mentioned earlier, without extra assumptions, it is impossible
to prove linear convergence rate of vanilla SGD. 

 There are two lines of research that
can prove linear convergence of SGD-type methods.
Along the first line, a few variance-reduction based methods
such as SVRG \cite{johnson2013accelerating}, SAGA \cite{defazio2014saga}
and SDCA \cite{shalev2013stochastic}
can achieve linear convergence.
Along the second line, with extra assumption
such as \textbf{WGC (Weak Growth Condition)},
vanilla SGD with constant stepsize can 
already achieve linear convergence \cite{vaswani2018fast}. 
This line of research is strongly motivated by
the interpolation assumption in machine learning
that the learner can fit the data, which is considered
a reasonable assumption in recent literature (e.g. \cite{vaswani2018fast,liu2018mass}).
Therefore, we focus on designing quantization algorithms
along the second line.






\begin{assumption} (WGC Assumption \cite{vaswani2018fast})
Suppose $f:\mathbb{R}^d\rightarrow\mathbb{R}, f(x)=\frac{1}{N}\sum\limits_{i=1}^N f_i(w)$ is the objective function. Stochastic ``functions''(algorithms) $\{\overline{\nabla}_i\}_{i\in[N]}$ satisfy \textbf{WGC} if
1) $\mathbb{E}[\overline{\nabla}_i (f_i,w)]=\nabla f_i(w)$, $\forall i\in[N], w\in\mathbb{R}^d$;
2) $\frac{1}{N}\sum\limits_{i=1}^N\mathbb{E}_{\overline{\nabla}_i}\left[\|\overline{\nabla}_i f_i(w)\|^2\right]\leq 2\rho L(f(w)-f(w^*))$.
\end{assumption}

To adapt the frequent-communication small-memory, we introduce DEED-SGD. The only differences between DEED-GD and DEED-SGD are 1) we use $\overline{\nabla}_i$ instead of the accurate gradient; 2) we use different quantization level. The full description of  DEED-SGD is given in the appendix.


\subsection{Infrequent-communication}\label{sec: Quantization in Federated Learning}
A main area in distributed optimization with infrequent-communication is Federated Learning (FL), which involves training models over remote devices or data centers, 
such as mobile phones or hospitals, and keeping the data localized due to privacy concern or communication efficiency \cite{li2019federated, lim2020federated}. In FL, some computation nodes might no have full participation in the updates and the data sets are non-iid. 
We remark that infrequent-communication is a generic design choice,
and can be used in a data-center setting as well. 
Although existing works like QSGD do not explore this degree of freedom,
infrequent communication can be combined with QSGD as well. 
Nevertheless, the theoretical benefit of the combination was not understood before
(partially because the total number of bits was not a focus of previous works,
and linear convergence rate was derived only recently \cite{mishchenko2019distributed}).


A classical algorithm in FL is Federated Averaging algorithm (FedAvg) which performs local stochastic gradient descent on computation nodes for every $E$ iterations with a server that performs model averaging \cite{mcmahan2016communication}.
Although there have been much efforts developing convergence guarantees for FedAvg, \cite{zhou2017convergence, stich2018local, wang2018cooperative, khaled2019first, yu2019parallel, woodworth2018graph,li2019convergence}, 
there is relatively scarce theoretical results on the combination
of FedAvg and quantization \cite{reisizadeh2019fedpaq, khaled2019first }. \cite{reisizadeh2019fedpaq, khaled2019first} either make unrealistic assumptions or only perform quantization on computation nodes, and thus they are not efficient as our double encoding scheme. 

We propose an algorithm called DEED-Fed. The difference between DEED-GD and DEED-Fed is that in FEED-Fed, the maximal error at iteration $k$ is proportional to learning rate $\eta_k$. Due to space limitation, a detailed comparison between the three proposed algorithms are given in the appendix.

\subsection{Theoretical Analysis}\label{subsec: theoratical analysis}
In this section, we give the computational and communication complexity of the algorithms DEED-GD, DEED-SGD and DEED-Fed. Since all of them are in the same framework, their proofs and results are similar. We put it into one single theorem.

\begin{theorem}\label{thm: three theorems}
Consider solving Problem \eqref{main problem} under one of the three settings by the corresponding algorithms (DEED-GD, DEED-SGD or DEED-Fed). Assume $f_i$ is $L$-smooth and $f$ is $\mu$ convex. Assume all $f_i$'s are $\mu$ convex in \textbf{DEED-Fed}. Denote $w_t$ as the iterate at iteration $t$ and $w^*$ is the optimal solution of Problem \eqref{main problem}.

For DEED-GD, we choose the learning rate $\eta_t\equiv\eta=\frac{2}{L+\mu}$, $c:=1-\eta\mu$, $c<c'<1$, and the maximal error at iteration $t$ is $s{c'}^{t+1}/2$ where $s$ is the quantization level.

For DEED-SGD, we assume (Weakth Growth Condition) WGC is satisfied for approximate gradient $\overline{\nabla}_i$ for every $f_i$ with parameter $\rho$. We choose the learning rate $\eta_t\equiv\eta=\frac{1}{\rho L}$, $c:=1-\eta\mu$, $c<c'<1$, and the maximal error at iteration $t$ is $\sqrt{s{c'}^{k+1}}/2$ and the error is unbiased.

For DEED-Fed, we choose the learning rate $\eta_t:=\frac{\beta}{t+\gamma}$ for some $\beta>\frac{1}{\mu}$, $\gamma>1$ such that $\eta_0\leq\frac{1}{4L}$ and $\eta_t\leq 2\eta_{t+E}$. Let the maximal error at iteration $t\in\{0,E,2E,\cdots\}$ be $s\eta_t$.

We have the following results:

\begin{itemize}
\item DEED-GD communicates $\tilde{O}(Nd)$ bits at iteration $t\geq1$, and
\begin{eqnarray}
\|w_t-w^*\|\leq ({c'})^t\left(\max\left\{0, \|w_0-w^*\|-\frac{c\eta s}{c'-c}\right\} + \frac{c'\eta s}{c'-c}\right).\notag
\end{eqnarray}
\item DEED-SGD communicates $\tilde{O}(Nd)$ bits at iteration $t\geq1$, and
\begin{eqnarray}
\mathbb{E}\|w_t-w^*\|^2\leq ({c'})^t\left(\max\left\{0, \|w_0-w^*\|^2-\frac{c\eta^2 s}{c'-c}\right\} + \frac{c'\eta^2 s}{c'-c}\right).\notag
\end{eqnarray}
\item DEED-Fed communicates
$\tilde{O}(Nd)$ bits at iteration $t\in\{E,2E,\cdots\}$, and
\begin{eqnarray}
\mathbb{E}\|\overline{w}_{t}-w^*\|^2\leq\frac{v}{\gamma+t},\notag
\end{eqnarray}
where $v$ is some constant dependent on the Federated learning settings (e.g. full participant or not) as well as the the initial error $\|w_0-w^*\|^2$.

\end{itemize}
\end{theorem}


Remark 1: Based on these results, we can easily
compute the total number of bits needed to achieve
a certain accuracy; see Table \ref{table: GD} and Table \ref{table: SGD}. 

Remark 2: Our result allows to trade-off communication time
and computation time. 
By changing the parameters $c'$ and $s$, we can find optimal choice of convergence speed and error size.

Theorem \ref{thm: three theorems} implies that to achieve $\|w_T-w^*\|\leq\varepsilon$, we need $\tilde{O}(\kappa\log\frac{1}{\varepsilon})$ iterations for DEED-GD and DEED-SGD and $O(\frac{1}{\varepsilon^2})$ iterations for DEED-Fed. These convergence rates match those
of the corresponding algorithms with infinite bandwidth. 
Due to space limitation, we eliminate the detailed definitions of some constants in Theorem \ref{thm: three theorems} and we will provide the details in the appendix.

\section{Quantization of Nesterov acceleration}\label{sec: Quantization of Nesterov acceleration}
In frequent-communication large-memory setting, we combine Nesterov's 
acceleration with our quantization scheme DEED. The accelerating algorithm is very similar to Algorithm \ref{our_gd}. The only difference between Algorithm \ref{our_gd} and this accelerated version is that we add momentum in the final update step. The full description of the algorithm is given in the appendix.

\begin{theorem}\label{our_nes_thm}
Consider solving Problem \ref{main problem} by Algorithm \textbf{A-DEED-GD} and assume each $f_i$ is $L$-smooth and $f$ is $\mu$-strongly convex .
Let the learning rate $\eta=\frac{1}{L}$, the constant $c:=\sqrt{1-\sqrt{\frac{\mu}{L}}}$ such that $c<c'<1$, and we do quantization with maximal error $s{c'}^{k+1}/2$ at every iteration $k$.
Then we have:
\begin{itemize}
    \item[(1)] $
\|x_k-x^*\|\leq\sqrt{\frac{2}{\mu}}\cdot {c'}^k\cdot\sqrt{\Delta/\gamma^{2k} + C}$, where $C=\beta_s^2+\alpha_s+\beta\sqrt{\alpha_s}-\Delta$, $\alpha_s=\frac{s^2}{L(\gamma^2-1)}+\Delta$, $\beta_s=\left(\frac{3\sqrt{\frac{2}{L}}+5\sqrt{\frac{2}{\mu}}}{c(\gamma^2-1)}\right)s\gamma$, $\gamma=c'/c>1$.
    \item[(2)] The number of bits at iteration $k\geq1$ is $O\left(Nd\right)$.
\end{itemize}
\end{theorem}

Theorem \ref{our_nes_thm} implies that we can improve the linear convergence rate in frequent-communication large-memory setting from $O(\kappa)$ to $O(\sqrt{\kappa})$ with acceleration trick, where $\kappa$ is the condition number of the objective function. This also provides a fewer total number of bits in communication. 

Remark 1: We separate algorithm \textbf{A-DEED-GD} out from previous three algorithms because we cannot directly use theorem \ref{thm: general convergence} due to the momentum. But the intuition and technique are very similar. Hence we put it in the \textbf{DEED} series.

Remark 2: We noticed an independent work \cite{li2020acceleration} which 
also proved an accelerated rate, but differs from our work in the following aspects. First, their encoding scheme is a non-trivial 
combination of the Nesterov's momentum and DIANA (which is why a separate paper
is written), while our combination is rather straightforward.
Second, their bound has an extra constant dependent on the communication scheme
(can be $N$ or number of encoding bits) while our bound does not.
Third, our work aims to develop a general framework, and acceleration
is just one case; while their work focused on acceleration
 in large-memory frequent-communication setting.

\section{Experiment}\label{sec: experiments}

\begin{figure}
\centering
\subfigure[Compare convergence speed]
{\label{fig:linear regression loss}\includegraphics[width=65mm]{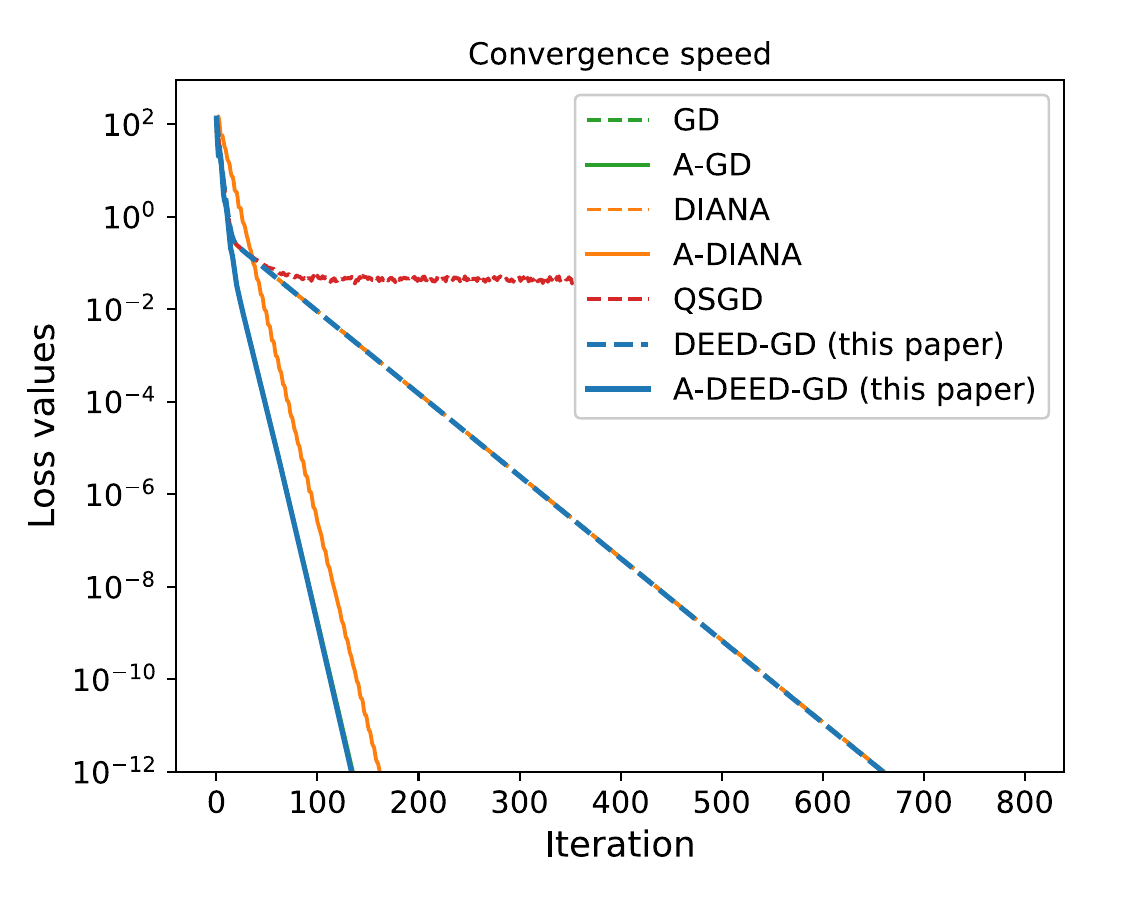}}
\subfigure[Compare communication efficiency]
{\label{fig:linear regression efficiency}\includegraphics[width=65mm]{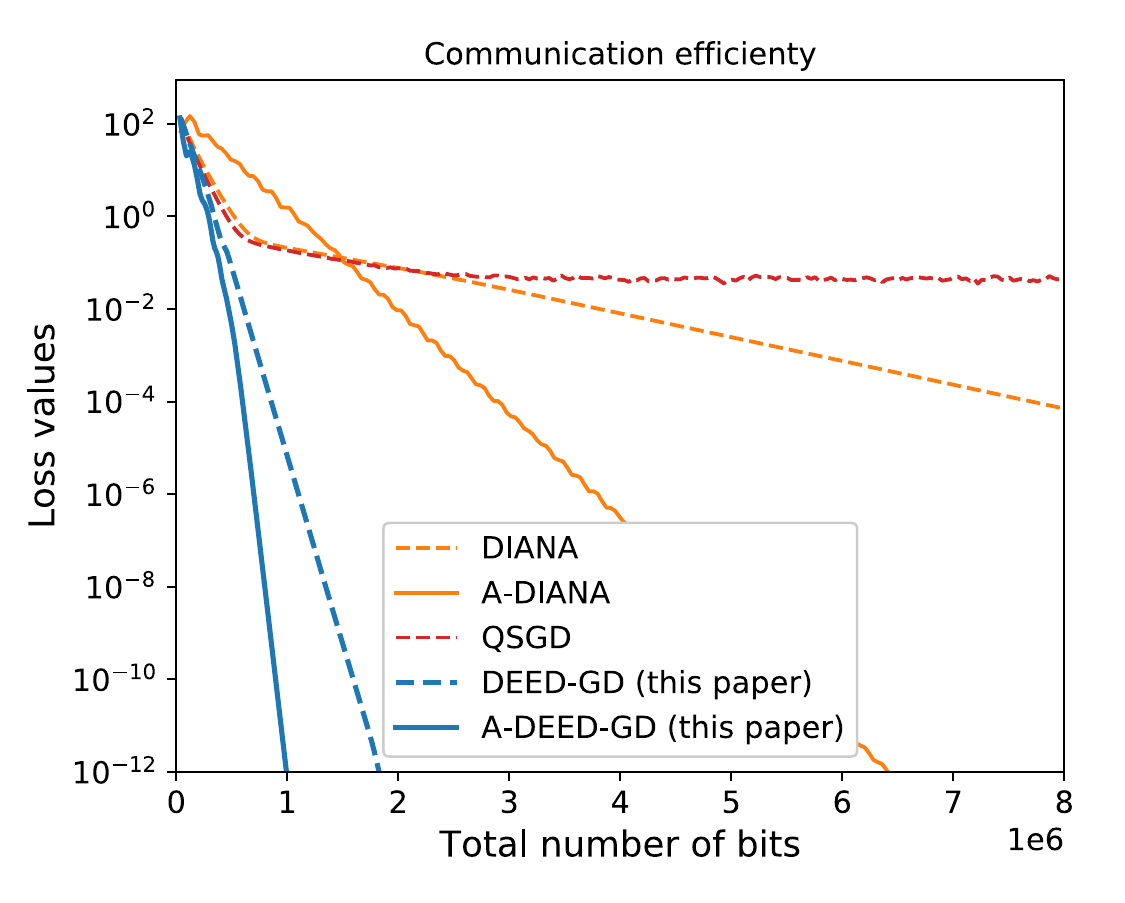}}
\caption{Results on linear regression experiments}
\end{figure}

\textbf{Linear regression}. We empirically validate our approach on linear regression problem as shown in Figure \ref{fig:linear regression loss} and Figure \ref{fig:linear regression efficiency}. 
The solid lines correspond to gradient descent type of algorithms and the dashed lines correspond to the accelerated versions. 
In Figure \ref{fig:linear regression loss}, the curves of our method coincide with the curves of the un-quantized baseline methods (GD and A-GD). 
QSGD performs the worst since it uses constant error. 
In Figure \ref{fig:linear regression efficiency}, our accelerated method achieves high-accuracy solution with the fewest number of bits, and another state-of-art algorithm takes more than 6 times of bits than ours to reach the same accuracy. Even without acceleration, our algorithm (DEED-GD) takes fewer bits than A-DIANA. Overall, our algorithms save the most number of bits in communication without scarifying the convergence speed.

\textbf{Image Classification}. We also compare our algorithm with other state-of-art algorithms (e.g. QSGD, TernGrad, DoubleSqueeze, DIANA) on image classification tasks on MNIST data set \cite{lecun1998gradient}. The results still show that our algorithms outperform others both in terms of convergence speed and communication complexity. The details of the two experiments are provided in the appendix.

\section{Conclusion}
In this paper, we provide a general convergence analysis for inexact gradient descent algorithms using absolute errors, that is tailored for quantized gradient methods.
Using this general convergence analysis, we derive a quantization scheme named DEED and 
propose algorithms for three common settings in distributed optimization:
frequent-communication large-memory, frequent-communication small-memory, and infrequent-communication (both large-memory and small-memory included). 
We also combine DEED with Nesterov's acceleration to provide an accelerated algorithm A-DEED-GD for frequent-communication large-memory, which improves the convergence rate from $O(\kappa)$ to $O(\sqrt{\kappa})$.
Our proposed algorithms converge almost as fast as their non-quantized versions and save communication in terms of bits. 
We empirically test our algorithms on linear regression problems and image classification tasks, and find that they use fewer bits
 than other algorithms. 



\newpage

\appendix

\section{Outline, Definition and Related Work}\label{Appendix: outline}
In the appendix, we first discuss some common definitions and related work. In Section \ref{Appendix: proof in general convergence analysis}, we provide the proof appeared in the general convergence analysis in Section \ref{sec: general framework}. In Section \ref{Appendix: algorithms}, we give the detailed description of the three algorithms based on our framework DEED. In Section \ref{Appendix: theorem in general quantizaiton scheme}, we provide the proof of the theorems for DEED in the three settings. In Section \ref{Appendix: experiments}, we illustrate the efficiency of our algorithms in linear regression problem and image classification tasks on MNIST dataest.

\begin{definition}
A differentiable function $ h :\mathbb{R}^d\rightarrow\mathbb{R}$ is $L$-smooth if $\enspace \forall x, y\in\mathbb{R}^d$,
\begin{eqnarray}
\left| h(y)- h(x)-\left\langle y - x, \nabla h(x)\right\rangle\right| \leq \frac{L}{2}\|x-y\|^2
\end{eqnarray}
\end{definition}

\begin{definition}
A differentiable function $f:\mathbb{R}^d\rightarrow\mathbb{R}$ is $\mu$-strongly-convex if $\enspace  \forall x, y\in\mathbb{R}^d$,
\begin{eqnarray}
f(y)-f(x)-\left\langle y - x, \nabla f(x)\right\rangle\geq\frac{\mu}{2}\|x-y\|^2
\end{eqnarray}
\end{definition}

\begin{definition}\label{condition number}
Suppose $f:\mathbb{R}^d\rightarrow\mathbb{R}, f(w):=\frac{1}{N}\sum\limits_{i=1}^N f_i(w)$, where $f_i$ is $L$-smooth $\forall i\in[N]$ and $f$ is $\mu$-strongly-convex. Then the condition number $\kappa$ of this collection of functions is
defined as $ \kappa = \frac{L}{\mu}$.
\end{definition}

\begin{assumption}\label{main assumption}
Assume that in (\ref{main problem}) $f$ is $\mu$-strongly-convex and $f_i$ is $L$-smooth $\forall i\in[N]$.
\end{assumption}


\textbf{Related work.} The study of the communication complexity in terms of bits for convex minimization of the problem \eqref{main problem} can be traced back to a classical work \cite{tsitsiklis1987communication} in 1987. 
This work focuses on the two-nodes case for frequent-communication large-memory setting, and proposed a nearly optimal algorithm using quantized gradient differences.
For multiple-nodes cases, \cite{mishchenko2019distributed} provides a linear convergence rate also using gradient differences on computing nodes.
In frequent-communication small-memory setting, 
to save more bits in communication, \cite{khirirat2018distributed, tang2019doublesqueeze} consider double encoding on gradient differences. In infrequent-communication setting, \cite{reisizadeh2019fedpaq} uses quantized gradient differences and proves sublinear convergence rate. Our work is different, as we provide convergence analysis on three settings for multiple-node cases and use doubling encoding to save more bits.

\section{Proof of Theorems for General Convergence Analysis}\label{Appendix: proof in general convergence analysis}
\begin{lemma}
Given a set $S=\left\{\left.x\in\mathbb{R}^d\right| \|w\|_2\leq M\right\}$, any (random) quantization algorithm that encoding a vector in $S$ by absolute error $\sigma$
takes at least $\left\lceil d\log_2 \frac{1}{\varepsilon}\right\rceil$ (in expectation) number of bits, where $\epsilon = \frac{\sigma}{M}$ . In addition, there exists a (random) algorithm that takes only $\left\lceil 1.05d + d\cdot\log_2 \frac{1+2\varepsilon}{\varepsilon}\right\rceil$ bits (in expectation).
\end{lemma}

\textbf{Proof sketch.} For the lower bound, we only need to prove the deterministic version since every random algorithm can be reduced to a deterministic algorithm by fixing $\xi$. Then it is equivalent to cover $S$ with small balls and proof follows. And we will use constructive method to prove the upper bound.

\begin{proof}
We first prove the lower bound. $\forall m\in \mathbf{E}(S)$, construct a ball centered at $\mathbf{D}(m)$ with radius $\sigma$. Then all these balls form a cover of $S$. Otherwise there is a vector $v\in S$ outside the cover and
\begin{eqnarray}
\left\|v - \mathbf{D}(\mathbf{E}(v))\right\|_2 \geq\min\limits_{m\in \mathbf{E}(S)}\left\|v - \mathbf{D}(m)\right\|_2>\sigma,\notag
\end{eqnarray}
which contradicts to the assumption.

Hence, the sum of the volumes of these small balls is not less than the volume of $S$. Finally, $|\mathbf{E}(S)|\geq\frac{M^d}{\sigma^d}$ follows.

On the other hand, we divide the whole space by cubes with side length $\frac{2\sigma}{\sqrt{d}}$ regularly. Then select the cubes who have non-empty intersection with $S$. Then every point in $S$ is contained in a cube, which means it can be encoded (unbiasedly) by the vertices of the cube with maximal error $\sigma$. Then these cubes must be contained in ball $B(0,M+2\sigma)$ since the diameter of the cube is $2\sigma$. In this case, the number of cubes is at most $\frac{\frac{\pi^{\frac{d}{2}}}{\Gamma\left(\frac{d}{2}+1\right)}(M+2\sigma)^d}{\left(\frac{2\sigma}{\sqrt{d}}\right)^d}$, and then the number of bits is at most $\left\lceil\log_2\frac{\pi^{\frac{d}{2}}}{\Gamma\left(\frac{d}{2}+1\right)}\left(\frac{(1+2\varepsilon)\sqrt{d}}{2\varepsilon}\right)^d\right\rceil$.

Recall Stirling's formula $\Gamma(n+1)\geq\sqrt{2\pi n}\left(\frac{n}{e}\right)^n$, we have
\begin{eqnarray}
\log_2\frac{\pi^{\frac{d}{2}}}{\Gamma\left(\frac{d}{2}+1\right)}\left(\frac{(1+2\varepsilon)\sqrt{d}}{2\varepsilon}\right)^d &\leq& \log_2\frac{\pi^{\frac{d}{2}}}{\sqrt{\pi d}\left(\frac{d}{2e}\right)^\frac{d}{2}}\left(\frac{(1+2\varepsilon)\sqrt{d}}{2\varepsilon}\right)^d\notag\\
&=& d\cdot\log_2 \frac{\sqrt{2\pi e}}{(\pi d)^{\frac{1}{2d}}}\frac{1+2\varepsilon}{2\varepsilon}\notag\\
&\leq& d\cdot\log_2 \sqrt{\frac{\pi e}{2}}\frac{1+2\varepsilon}{\varepsilon}\notag\\
&\leq&  1.05d + d\cdot\log_2 \frac{1+2\varepsilon}{\varepsilon}.\notag
\end{eqnarray}
\end{proof}

\begin{theorem}
Suppose Assumption \ref{assumption: Ft} holds and $\{w_t\}$ is a sequence generated by \begin{align}\label{eqn: general iterate update }
    w_{t+1}=F_t(w_t)+e_t,
\end{align}
for some chosen initial value $w_0$ and $e_t$ is a zero-mean random  noise depending on the (iteration) history and is bounded by $\alpha_t$. Define series $C_k^2=\sum\limits_{i=0}^{k-1}\alpha_i^2\prod\limits_{j=i+1}^{k-1}c_j^2$ and $D_k^2=\prod\limits_{i=0}^{k-1}c_i^2$.
Then we have 
\begin{eqnarray}\label{inequ_general_convergence}
\mathbb{E}\left[\|w_{T}-w^*\|^2\right]\leq D_T^2\|w_0-w^*\|^2+C_T^2.
\end{eqnarray}
In addition, there exists functions series $\{F_i\}_{i\geq0}$ and 
 noise $ \{ e_t \}_{t\geq 0} $ to make the inequality hold. Besides, if we suppose the sequence of the Lipschitz constants $\{c_i\}$ is non-decreasing, then the right hand side of (\ref{inequ_general_convergence}) converges linearly if and only if all $c_t$'s are always bounded above by a constant $c<1$ and $\alpha_t$ converges to 0 linearly.
\end{theorem}

\begin{proof}
The inequality (\ref{inequ_general_convergence}) is straightforward.
\begin{eqnarray}
\mathbb{E}\left[\|w_{t+1}-w^*\|^2|w_{t}\right]&=&\mathbb{E}\left[\|F_{t}(w_t)+e_t-w^*\|^2|w_{T-1}\right]\notag\\
&=& \mathbb{E}\left[\|F_{t}(w_t)-w^*\|^2|w_{T-1}\right] + \alpha_t^2\notag\\
&\leq& c_t^2\|w_t-w^*\|^2+\alpha_t^2.\notag
\end{eqnarray}

Then we can prove inequality (\ref{inequ_general_convergence}) by mathematical induction. For $T=0$, we have $\|w_0-w^*\|^2\equiv\|w_0-w^*\|^2$. Suppose it holds for $T\leq k$, we have
\begin{eqnarray}
\mathbb{E}\|w_{k+1}-w^*\|^2&\leq& c_k^2\mathbb{E}\|w_k-w^*\|^2 + \alpha_k^2\notag\\
&\leq& c_k^2\left(D_k^2\|w_0-w^*\|^2+C_k^2\right)+\alpha_k^2\notag\\
&=& D_{k+1}\|w_0-w^*\|^2+C_{k+1}^2.\notag
\end{eqnarray}
The inductions succeed.

Next, we define $F_t(x):=c_t x$. Given history $\{w_0,w_2,\cdot,w_t\}$, we assign $\overline{e}_t$ as an arbitrary vector orthogonal to $F_t(w_t)-w^*$ with length $\alpha_t$. Define $e_t=\overline{e}_t$ with probability $\frac{1}{2}$, and $e_t=-\overline{e}_t$ otherwise. In this case, we always have
\begin{eqnarray}
\|w_{t+1}-w^*\|^2 = c_t^2\|w_t-w^*\|^2+\alpha_t^2,\notag
\end{eqnarray}
and 
\begin{eqnarray}
\|w_{k+1}-w^*\|^2= D_{k+1}\|w_0-w^*\|^2+C_{k+1}^2\notag
\end{eqnarray}
follows.

Finally, suppose the sequence of the Lipschitz constants $\{c_i\}$ is non-decreasing.
\begin{itemize}
    \item \textbf{Necessity}. Suppose there exists constant $C,M>0$ and $c\in(0,1)$ such that $D_T^2\|w_0-w^*\|^2+C_T^2\leq Cc^T$, $\forall T\geq M$.
    
    We firstly prove a lemma.
    \begin{lemma}
    The function $g(x):=\frac{\ln\frac{1}{1-x}}{x}$ is increasing on $(0,1)$.
    \end{lemma}
    \begin{proof}
    $\forall x\in(0,1)$, we have $g'(x)=\frac{\frac{x}{1-x}+\frac{\ln\frac{1}{1-x}}{x^2}}{x^2}>0$.
    \end{proof}
    Recall $c_t\geq c_0$, $\forall t\geq0$, we have $\frac{\ln\frac{1}{c_t}}{1-c_t}\leq\frac{\ln\frac{1}{c_0}}{1-c_0}$, i.e. $\ln\frac{1}{c_t}\leq C_0(1-c_t)$ where $C_0:=\frac{\ln\frac{1}{c_0}}{1-c_0}$. Because $D_T^2\leq\frac{C}{\|w_0-w^*\|^2}c^T$, we have $\sum\limits_{i=0}^{T-1} 2\ln\frac{1}{c_i}\geq\ln\frac{\|w_0-w^*\|^2}{C}+T\ln\frac{1}{c}$. Moreover, $2C_0\sum\limits_{i=0}^{T-1} (1-c_i)\geq\ln\frac{\|w_0-w^*\|^2}{C}+T\ln\frac{1}{c}$. This suggests $\forall k\geq0$, $\lim\limits_{T\rightarrow\infty}\frac{\sum\limits_{i=k}^{T-1} (1-c_i)}{T}\geq\frac{\ln\frac{1}{c}}{2C_0}$, and $1-c_k\geq\frac{\ln\frac{1}{c}}{2C_0}$ follows. On the other hand, $Cc^T\geq C_T^2\geq\alpha_{T-1}^2$. Hence, $c_i$ is bounded by $c:=1-\frac{\ln\frac{1}{c}}{2C_0}$ and $\alpha_i$ diminishes exponentially.
    \item \textbf{Sufficiency}. Suppose $c_i\leq c<1, \forall i\geq0$ and $\alpha_t\leq C\alpha^t,\forall t\geq M$ for constant $C>0,M\geq0$. Without loss of generality, we assume $M=0$. We only need to prove $D_k$ and $C_k$ diminish exponentially separately. It is trivial for $D_k$. For $C_k$, we have $C_k^2:=C_k^2=\sum\limits_{i=0}^{k-1}\alpha_i^2\prod\limits_{j=i+1}^{k-1}c_j^2\leq Ck\beta^k$ where $\beta=\max\{c,\alpha\}$. Then $C_k$ diminishes exponentially since $Ck\beta^k\leq\beta^{\frac{k}{2}}$ for sufficient large $k$.
\end{itemize}
\end{proof}

\section{Algorithms}\label{Appendix: algorithms}

\subsection{DEED-GD}
The pseudo-code of DEED-GD is given in Algorithm \ref{our_gd}.

\subsection{DEED-SGD}
The algorithm of DEED-SGD is given below (algorithm \ref{our_sgd}). As we promised previously, there are only two differences between DEED-GD and DEED-SGD.
\begin{itemize}
    \item In line 5, we use approximate gradient $\overline{\nabla}_i$ instead of gradient $\nabla$.
    \item In line 2, line 6 and line 11, we use different $\eta$, $c$ and $c'$. And the maximal error in quantization is changed.
\end{itemize}

\begin{algorithm}[h]
\SetAlgoLined
\begin{algorithmic}[1]
 \STATE{Initialization: Each server $i\in[N]$ holds $w_0=s_{-1}^i=v_{-1}=0$, server 0 holds $v_{-1}=0$, $k=0$\;}
 \STATE{$\eta=\frac{1}{\rho L}$, $c=1-\frac{\mu}{\rho L}$, $c<c'<1$ and $s\in\mathbb{R}_+$ is the quantization level\;}
 \WHILE{the precision is not enough}
  \FOR{$i\in[N]$} \STATE{
  server $i$ computes $g_k^i=\overline{\nabla}_i f_i(x_k)$\;}
  \STATE{server $i$ does quantization $d_k^i=Q(g_k^i - s_{k-1}^i, \sqrt{s{c'}^{k+1}}/2)$\;}
  \STATE{server $i$ updates $s_k^i = d_k^i + s_{k-1}^i$\;}
  \STATE{server $i$ send $d_k^i$ to server 0\;}
  \ENDFOR
  \STATE{server 0 computes $s_k = \frac{1}{N}\sum\limits_{i=1}^N d_k^i + s_{k-1}$\;}
  \STATE{server 0 does quantization $u_k = Q(s_k - v_{k-1}, \sqrt{s{c'}^{k+1}}/2)$\;}
  \STATE{server 0 sends $u_k$ to server $i,\forall i\in[N]$\;}
  \STATE{server 0 updates $v_k = u_k + v_{k-1}$\;}
  \FOR{$i\in[N]$} \STATE{
  server $i$ updates $v_k = u_k + v_{k-1}$\;}
  \STATE{server $i$ updates $w_{k+1}= w_k - \eta v_k$\;}
  \ENDFOR
 \STATE{$k = k + 1$\;}
 \ENDWHILE
\end{algorithmic}
\caption{Double Encoding and Error Diminishing for Stochastic Gradient Descent (DEED-SGD)}
\label{our_sgd}
\end{algorithm}

\subsection{DEED-Fed}
We first introduce the original FedAvg algorithm. Define $F(w):=\sum\limits_{i=1}^N p_i F_i(w)$ where $F_i$'s are $\mu$-convex and $L$-smooth functions defined on $\mathbb{R}^d$ and $p_i\geq0, \sum\limits_{i=1}^N p_i=1$.
In each round, say round $k\geq0$, the center server sends weight $w_{tE}$ to $N$ slave nodes, and the $k^\text{th}$ slave nodes performs $E$ local updates (for $i$ in $\{0, 1, \cdot, E-1\}$):
\begin{eqnarray}
w_{tE+i+1}^k = w_{tE+i}^k - \eta_{tE+i}\nabla F_{k}(w_{tE+i}^k, \xi_{tE+i}^k).\notag
\end{eqnarray}

Finally, in \textbf{full participant setting}, all slave nodes sends their final weights to the center server, and center server computes
\begin{eqnarray}
w_{(t+1)E}:=\sum\limits_{i=1}^N p_i w_{(t+1)E}^i.\notag
\end{eqnarray}

For \textbf{partial participant setting}, not all slave nodes stay active in each round. Here are two more detailed settings.
\begin{itemize}
    \item[1)] In setting 1, we defined a set $S_{t+1}$ of $K$ indices selected randomly with replacement from $[N]$ with probability distribution $(p_1, \cdots, p_N)$. Then the center server updates
    \begin{eqnarray}\label{partial_setting_1}
    w_{(t+1)E}:=\frac{1}{K}\sum\limits_{i\in S_{t+1}}w_{(t+1)E}^i.
    \end{eqnarray}
    \item[2)] In setting 2, we defined a set $S_{t+1}$ of $K$ indices selected evenly and randomly without replacement from $[N]$. Then the center server updates
    \begin{eqnarray}\label{partial_setting_2}
    w_{(t+1)E}:=\frac{N}{K}\sum\limits_{i\in S_{t+1}}p_i w_{(t+1)E}^i.
    \end{eqnarray}
\end{itemize}

Except the assumption of smoothness and convexity, there are two more assumptions.

\begin{assumption}\label{fed_assump_1}
Let $\xi_t^k$ be sampled from the $k^\text{th}$ device's local data uniformly at random. The
variance of stochastic gradients in each device is bounded: $\mathbb{E}\|\nabla F_k(w_t^k, \xi_t^k)-\nabla F_k(w_t^k)\|^2\leq\sigma_k^2$.
\end{assumption}

\begin{assumption}\label{fed_assump_2}
The expected squared norm of stochastic gradients is uniformly bounded, i.e. $\mathbb{E}\|\nabla F_k(w_t^k, \xi_t^k)\|^2\leq G^2$.
\end{assumption}

Based on these two assumptions, their theorems said

\begin{theorem}\label{fedavg_thm}
With the algorithm above, we have
\begin{eqnarray}
\mathbb{E}\|\overline{w}_{t+1}-w^*\|^2\leq(1-\eta_t\mu)\mathbb{E}\|\overline{w}_t-w^*\|^2+\eta_t^2(B+C),\notag
\end{eqnarray}
where $B=\sum\limits_{k=1}^N p_k^2\sigma_k^2+6L\Gamma+8(E-1)^2G^2$, and $C$ is a constant depending on different setting. In full participant setting, $C=0$. In partial participant setting 1, $C=\frac{4}{K}E^2G^2$, and in partial participant setting 2, $C=\frac{N-K}{N-1}\frac{4}{K}E^2G^2$.
\end{theorem}

The algorithm \textbf{DEED-Fed} is a simple combination of DEED-GD and Federated Averaging algorithms. Algorithm \ref{our_fed} is the pseudo-code of fully-participant setting.

\begin{algorithm}[ht]
\caption{Double Encoding and Error Diminishing Federated Averaging (DEED-Fed)}
\begin{algorithmic}[1]
\STATE Initialization: Each server $i\in[N]$ holds $w_0=s_{-E}^i=v_{-E}=0$, server 0 holds $v_{-E}=0$, $k=0$;
\STATE Hyper-parameters:  $\eta_k\in\mathbb{R}_+$, parameter $s \in \mathbb{R}_+ $;
\WHILE{the precision is not enough}
    \FOR{$i\in[N]$}
        \STATE  $w_{k+1}^i = w_k^i - \eta_k\nabla f_i(w_k^i,\xi_k^i)$\;
        \STATE $k = k + 1$\;
    \ENDFOR
    \IF{$E|k$}
        \FOR{$i\in[N]$}
            \STATE   server $i$ does quantization $d_k^i=Q(w_k^i - s_{k-E}^i, \frac{s\eta_k}{2})$\;
            \STATE   server $i$ updates $s_k^i = d_k^i + s_{k-E}^i$\;
            \STATE  server $i$ send $d_k^i$ to server 0\;
        \ENDFOR
        \STATE   server 0 computes $s_k = \sum\limits_{i=1}^N p_id_k^i + s_{k-E}$\;
        \STATE   server 0 does quantization $u_k = Q(s_k - v_{k-E}, \frac{s\eta_k}{2})$\;
        \STATE   server 0 sends $u_k$ to server $i,\forall i\in[N]$\;
        \STATE   server 0 updates $v_k = u_k + v_{k-E}$\;
        \FOR{$i\in[N]$}
            \STATE   server $i$ updates $v_k = u_k + v_{k-E}$\;
            \STATE   server $i$ updates $w_{k}= v_k$\;
        \ENDFOR
    \ENDIF
\ENDWHILE
\end{algorithmic}
\label{our_fed}
\end{algorithm}

To change algorithm \ref{our_fed} into partial participant versions, we only need to replace $[N]$ by set $S$ in line 9 and change the summation in line 14 into (\ref{partial_setting_1}) or (\ref{partial_setting_2}) correspondingly.

\section{Theorems for DEED in Three Settings}\label{Appendix: theorem in general quantizaiton scheme}
In this section, we restate the theorem \ref{thm: three theorems} separately and give proof separately, too.

\subsection{Proof for DEED-GD}
\begin{theorem}
In algorithm \ref{our_gd}, we choose the learning rate $\eta_t\equiv\eta=\frac{2}{L+\mu}$, $c:=1-\eta\mu$, $c<c'<1$, and the maximal error at iteration $t$ is $s{c'}^{t+1}/2$ where $s$ is the quantization level. Then algorithm \ref{our_gd} communicates $\tilde{O}(Nd)$ bits at iteration $t\geq1$, and
\begin{eqnarray}\label{our_gd_convergence}
\|w_t-w^*\|\leq ({c'})^t\left(\max\left\{0, \|w_0-w^*\|-\frac{c\eta s}{c'-c}\right\} + \frac{c'\eta s}{c'-c}\right).
\end{eqnarray}
\end{theorem}

\textbf{Proof sketch.} First of all, according to mathematical induction and triangle inequality, we see that the effective at iteration $t$ is bounded by $\eta\cdot s{c'}^{t+1}$. Besides, consider the function $F_t(w)\equiv F(w):=w-\eta\nabla f(w)$. Because $f$ is $L$-smooth and $\mu$-convex, we know that $F$ is $\left(1-\eta\mu\right)$-Lipschitz with fixed point $w^*$. Then, according to our framework, we can prove that $\|w_t-w^*\|$ converges at speed $c'$. This is exactly (\ref{our_gd_convergence}).

To bound the number of bits produced by each communication, we only need to prove that $\|g_t^i-s_{t-1}^i\|=\Theta({c'}^t)$ and $\|s_t-v_{t-1}\|=\Theta({c'}^t)$ $\forall i\in[N]$ because of lemma \ref{quan_vect_lemma}.

Notice that these two norms are all close to $L\|w_t-w_{t-1}\|$, and $\|w_t-w_{t-1}\|$ can be bounded by $\|w_t-w^*\|+\|w^*-w_{t-1}\|$ which are also $\Theta({c'}^t)$ by (\ref{our_gd_convergence}) and everything's done. In our real proof, we will bound these terms carefully to give a tighter bound.

\begin{proof}
First of all, we can give a deterministic version of theorem \ref{thm: general convergence} as we promised below theorem \ref{thm: general convergence}.
\begin{proposition}\label{thm_general_deter}
Suppose $F:\mathbb{R}^d\rightarrow\mathbb{R}^d$ is a continuous function with Lipschitz constant $c<1$. Define a sequence $\{e_i\}_{i\geq1}$ in $\mathbb{R}^d$ satisfies $\|e_i\|\leq \eta s {c'}^{i}$ where $c<c'<1$. Then $\forall x_0\in \mathbb{R}^d$, the sequence constructed by $x_{t+1} = F(x_t) + e_{t+1}$ satisfies $D(x_t)\leq{c'}^t\left(\max\left\{0, D(x_0)-\frac{c\eta s}{c'-c}\right\} + \frac{c'\eta s}{c'-c}\right)$ where $D(w):=\|w-w^*\|$ and $w^*$ is the fixed point of $F$.
\end{proposition}

\begin{proof}
According to definition, we have the following inequalities.
\begin{eqnarray}
D(x_{k+1})&=&D(F(x_k)+e_{k+1})\notag\\
&\leq &D(F(x_k)) + \|e_{k+1}\|\notag\\
&\leq& c\cdot D(x_k) + \eta s {c'}^{k+1}.\notag
\end{eqnarray}

Hence,
\begin{eqnarray}
\frac{D(x_k)}{c^k}&\leq&\frac{D(x_{k-1})}{c^{k-1}}+\eta s\left(\frac{c'}{c}\right)^k\notag\\
&\leq&D(x_0)+\eta s\sum\limits_{i=1}^k\left(\frac{c'}{c}\right)^i\notag\\
&=&D(x_0)+\eta s\frac{\left(\frac{c'}{c}\right)^{k+1}-1}{\frac{c'}{c}-1}\notag
\end{eqnarray}

Finally, we have $D(x_k)\leq c^k\left(D(x_0)-\frac{c\eta s}{c'-c}\right)+{c'}^{k}\left(\frac{c'\eta s}{c'-c}\right)$ which is bounded by ${c'}^k\left(\max\left\{0, D(x_0)-\frac{c\eta s}{c'-c}\right\} + \frac{c'\eta s}{c'-c}\right)$.
\end{proof}

Then we can prove the inequality (\ref{our_gd_convergence}). Notice $\nabla F=I-\eta\nabla^2 f$ where $\mu I\preceq \nabla^2 f\preceq LI$, we see $F$ is $c:=1-\eta\mu$ Lipschitz. By proposition \ref{thm_general_deter}, we only need to prove $\|v_k-g_k\|\leq s{c'}^{k+1}$ where $g_k=\nabla f(x_k)$. By induction, we have $s_k-\frac{1}{N}\sum\limits_{i=1}^N s_k^i=\frac{1}{N}\sum\limits_{i=1}^N d_k^i + s_{k-1} - \frac{1}{N}\sum\limits_{i=1}^N (d_k^i + s_{k-1}^i) = s_{k-1}-\frac{1}{N}\sum\limits_{i=1}^N s_{k-1}^i = s_{-1}-\frac{1}{N}\sum\limits_{i=1}^N s_{-1}^i=0$. Then,
\begin{eqnarray}
\|v_k - g_k\|&=&\|u_k + v_{k-1} - g_k\|\notag\\
&\leq& \|u_k - (s_k - v_{k-1})\| + \|s_k - g_k\|\notag\\
&\leq& \|u_k - (s_k - v_{k-1})\| + \left\|\frac{1}{N}\sum\limits_{i=1}^N (d_k^i + s_{k-1}^i - g_k^i)\right\|\notag\\
&\leq& \|u_k - (s_k - v_{k-1})\| + \frac{1}{N}\sum\limits_{i=1}^N \left\|d_k^i + s_{k-1}^i - g_k^i \right\|\notag\\
&\leq& s{c'}^{k+1}.\notag
\end{eqnarray}

For convenience, we define constant $\xi_s = \max\left\{0, \|x_0-x^*\|-\frac{c\eta s}{c'-c}\right\} + \frac{c'\eta s}{c'-c}$. The convergence result comes directly from proposition \ref{thm_general_deter}.

To bound the number of bits, we only need to calculate the maximal norm of the vector we need to encode. Actually, $\forall i\in[N]$, $\forall k\geq 1$ we have
\begin{eqnarray}
\|g_k^i - s_{k-1}^i\| &\leq& \|g_k^i - g_{k-1}^i\| +\left\|g_{k-1}^i-s_{k-1}^i\right\|\notag\\
&\leq& L\|x_k-x_{k-1}\| + \frac{s {c'}^{k}}{2}\notag\\
&\leq& L\eta\|v_{k-1}\| + \frac{s {c'}^{k}}{2}\notag\\
&\leq& L\eta\|g_{k-1}\| + L\eta s{c'}^k + \frac{s {c'}^{k}}{2}\notag\\
&\leq& L^2\eta\|x_{k-1} - x^*\| + L\eta s{c'}^k + \frac{s {c'}^{k}}{2}\notag\\
&\leq& L^2\eta {c'}^{k-1}\xi_s + L\eta s{c'}^k + \frac{s {c'}^{k}}{2}\notag\\
&=& \left(\frac{L^2\eta\xi_s+L\eta sc'+\frac{sc'}{2}}{{c'}^2}\right){c'}^{k+1}.\notag
\end{eqnarray}

Similarly, we have
\begin{eqnarray}
\|s_k-v_{k-1}\|&\leq& \|s_k-g_k\|+\|g_k - g_{k-1}\|+\|g_{k-1}-v_{k-1}\|\notag\\
&\leq& L\|x_k-x_{k-1}\| + \frac{3s {c'}^{k}}{2}\notag\\
&\leq& \left(\frac{L^2\eta\xi_s+L\eta sc'+\frac{3sc'}{2}}{{c'}^2}\right){c'}^{k+1}.\notag
\end{eqnarray}

In this case, the error fraction (length divided by error, i.e. inverse of relative error) is bounded by $\zeta_s:=\frac{2L^2\eta\frac{\xi_s}{s}+2L\eta c'+3c'}{{c'}^2}$, and the number of bits is at most $\left(1.05+\log_2(\zeta_s+2)\right)d$ by lemma \ref{quan_vect_lemma}.
\end{proof}

\subsection{Proof for DEED-SGD}
\begin{theorem}
In algorithm \ref{our_sgd}, we assume (Weak Growth Condition) WGC is satisfied for approximate gradient $\overline{\nabla}_i$ for every $f_i$ with parameter $\rho$. We choose the learning rate $\eta_t\equiv\eta=\frac{1}{\rho L}$, $c:=1-\eta\mu$, $c<c'<1$, and the maximal error at iteration $t$ is $\sqrt{s{c'}^{k+1}}/2$ and the error is unbiased.

DEED-SGD communicates $\tilde{O}(Nd)$ bits at iteration $t\geq1$, and
\begin{eqnarray}\label{our_sgd_convergence}
\mathbb{E}\|w_t-w^*\|^2\leq ({c'})^t\left(\max\left\{0, \|w_0-w^*\|^2-\frac{c\eta^2 s}{c'-c}\right\} + \frac{c'\eta^2 s}{c'-c}\right).
\end{eqnarray}
\end{theorem}

\textbf{Proof sketch.} First of all, we introduce theorem 5 in \cite{vaswani2018fast} to show that with WGC, stochastic gradient descent converges linearly.

\begin{lemma}\label{sgd_lemma}
Suppose $f:\mathbb{R}^d\rightarrow\mathbb{R}$ is $\mu$-strongly-convex. Besides, $f(x)=\frac{1}{N}\sum\limits_{i=1}^N f_i(x)$, where each $f_i$ is $L$-smooth. We assume WGC is satisfied for approximate gradient $\overline{\nabla}_i$ with parameter $\rho$. Then series $\{x_i\}_{i\geq0}$ generated by iteration formula
\begin{eqnarray}
x_{k+1}:=x_k-\frac{1}{N}\sum\limits_{i=1}^N \overline{\nabla}_i f_i(x)
\end{eqnarray}
satisfy
\begin{eqnarray}
\mathbb{E}\left[\|x_{k+1}-x^*\|^2|x_k\right]\leq\left(1-\frac{\mu}{\rho L}\right)\|x_k-x^*\|^2.\notag
\end{eqnarray}
\end{lemma}

Being similar to the proof in previous section, we can bound the effective error by $\eta\sqrt{s {c'}^{k+1}}$. Then we have
\begin{eqnarray}
\mathbb{E}\|x_{k+1}-x_k\|^2\leq c\mathbb{E}\|x_k-x^*\|^2 + \eta^2 s{c'}^{k+1}.\notag
\end{eqnarray}

Then we can put it in our framework. By using the same technique in proposition \ref{thm_general_deter}, we have inequality (\ref{our_sgd_convergence}). Finally, we can bound the number of bits by using lemma \ref{quan_vect_lemma}.

\begin{proof}
As we showed in sketch, we have proved inequality (\ref{our_sgd_convergence}). We only need to prove that $\mathbb{E}\|g_k^i-s_{k-1}^i\|^2$ and $\mathbb{E}\|s_k-v_{k-1}\|^2$ are $O({c'}^k)$, $\forall i\in[N]$, because $\ln{x^2}=2\ln{x}$ is a concave function. By triangle inequality and the definition of $L$-smooth, we only need to prove $\mathbb{E}\|w_k-w_{k-1}\|^2\leq 2\mathbb{E}\|w_k-w^*\|^2 + 2\mathbb{E}\|w^*-w_{k-1}\|^2$ is $O({c'}^k)$, which is obvious by inequality ($\ref{our_sgd_convergence}$).
\end{proof}

\subsection{Proof for DEED-Fed}
\begin{theorem}
In algorithm \ref{our_fed}, we choose the learning rate $\eta_t:=\frac{\beta}{t+\gamma}$ for some $\beta>\frac{1}{\mu}$, $\gamma>1$ such that $\eta_0\leq\frac{1}{4L}$ and $\eta_t\leq 2\eta_{t+E}$. Let the maximal error at iteration $t\in\{0,E,2E,\cdots\}$ be $s\eta_t$. Then DEED-Fed communicates 
$\tilde{O}(Nd)$ bits at iteration $t\in\{E,2E,\cdots\}$, and
\begin{eqnarray}\label{our_fed_convergence}
\mathbb{E}\|\overline{w}_{t}-w^*\|^2\leq\frac{v}{\gamma+t},
\end{eqnarray}
where $v:=\max\left\{\frac{\beta^2(B+C+s^2)}{\beta\mu-1}, \gamma\|w_0-w^*\|^2\right\}$. Here, $B=\sum\limits_{k=1}^N p_k^2\sigma_k^2+6L\Gamma+8(E-1)^2G^2$, and $C$ is a constant depending on different setting. In full participant setting, $C=0$. In partial participant setting 1, $C=\frac{4}{K}E^2G^2$, and in partial participant setting 1, $C=\frac{N-K}{N-1}\frac{4}{K}E^2G^2$.
\end{theorem}

\textbf{Proof sketch.} Recall the theorem with no quantization in \cite{li2019convergence}.

\begin{theorem}\label{fedavg_thm}
Assume assumption \ref{fed_assump_1} and \ref{fed_assump_2} hold. For FedAvg, we have
\begin{eqnarray}\label{fed_convergence}
\mathbb{E}\|\overline{w}_{t+1}-w^*\|^2\leq(1-\eta_t\mu)\mathbb{E}\|\overline{w}_t-w^*\|^2+\eta_t^2(B+C),
\end{eqnarray}
where $B=\sum\limits_{k=1}^N p_k^2\sigma_k^2+6L\Gamma+8(E-1)^2G^2$, and $C$ is a constant depending on different setting. In full participant setting, $C=0$. In partial participant setting 1, $C=\frac{4}{K}E^2G^2$, and in partial participant setting 1, $C=\frac{N-K}{N-1}\frac{4}{K}E^2G^2$.

Moreover, with inequality (\ref{fed_convergence}), we have $\mathbb{E}\|\overline{w}_{t}-w^*\|^2\leq\frac{v}{\gamma+t}$
where $v:=\max\left\{\frac{\beta^2(B+C)}{\beta\mu-1}, \gamma\|w_0-w^*\|^2\right\}$.
\end{theorem}

Combine this theorem error analysis, we have\footnote{We can prove better inequality since the quantization is only done on iteration $E,2E,\cdots$. However, this is enough since it won't change the convergence rate.}
\begin{eqnarray}
\mathbb{E}\|\overline{w}_{t+1}-w^*\|^2\leq(1-\eta_t\mu)\mathbb{E}\|\overline{w}_t-w^*\|^2+\eta_t^2(B+C+s^2).\notag
\end{eqnarray}
We can already put the map from $\overline{w}_t$ to $\overline{w}_{t+1}$ into our framework, and prove that it converges sublinearly. Actually, we can just use theorem \ref{fed_convergence} and conclude that inequality (\ref{our_fed_convergence}) holds. The only difficulty is the bound for communication. We cannot bound $\|\overline{w}_t-\overline{w}_{t+E}\|$ by $\|\overline{w}_t-w^*\|+\|w^*-\overline{w}_{t+E}\|$ since it is $O(1/\sqrt{t})$, while the precision is $O(1/t)$. Please see this part in the proof below.

\begin{proof}
We have proved inequality (\ref{our_fed_convergence}). The only thing left is to prove that $\mathbb{E}\|w_{(t+1)E}^k-w_{tE}^k\|^2=O(1/t^2)$.

Notice that for each larger iteration $t\geq0$ and slave node $k$, we have
\begin{eqnarray}
\|w_{(t+1)E}^k-w_{tE}^k\|&\leq&\sum\limits_{i=0}^{E-1}\|w_{tE+i+1}^k-w_{tE+i}^k\|\notag\\
&\leq& \sum\limits_{i=0}^{E-1}\eta_{tE+i}\|\nabla F_k(w_{tE+i}^k, \xi_{tE+i}^k)\|\notag\\
&\leq& \eta_{tE}\sum\limits_{i=0}^{E-1}\|\nabla F_k(w_{tE+i}^k, \xi_{tE+i}^k)\|.\notag
\end{eqnarray}

Hence,
\begin{eqnarray}
\mathbb{E}\|w_{(t+1)E}^k-w_{tE}^k\|^2&\leq& \eta_{tE}^2\cdot E\sum\limits_{i=0}^{E-1}\left(\mathbb{E}\|\nabla F_k(w_{tE+i}^k)\|^2+\sigma_k^2\right)\notag\\
&\leq& \eta_{tE}^2\cdot E\sum\limits_{i=0}^{E-1}\left(L^2\mathbb{E}\|w_{tE+i}^k-w^*\|^2 +\sigma_k^2\right)\notag\\
&\leq& \eta_{tE}^2\cdot E\sum\limits_{i=0}^{E-1}\left(2L^2\mathbb{E}\|w_{tE+i}^k-\overline{w}_{tE+i}\|^2 + 2L^2\mathbb{E}\|\overline{w}_{tE+i}-w^*\|^2 +\sigma_k^2\right)\notag\\
&\leq& \eta_{tE}^2\cdot E\sum\limits_{i=0}^{E-1}\left(2L^2\eta_{tE}^2E^2G^2 + 2L^2\frac{v}{\gamma + tE} +\sigma_k^2\right)\notag\\
&=& \eta_{tE}^2\cdot E^2\left(2L^2\eta_{tE}^2E^2G^2 + 2L^2\frac{v}{\gamma + tE} +\sigma_k^2\right).\notag
\end{eqnarray}

With the approximation above, all the vectors we need to do quantization are bounded by $$\sqrt{\eta_{tE}^2s^2+\eta_{tE}^2\cdot E^2\left(2L^2\eta_{tE}^2E^2G^2 + 2L^2\frac{v}{\gamma + tE} +\sigma_k^2\right)}$$ in expectation, and proof follows.
\end{proof}

\section{Algorithm A-DEED-GD and its Convergence Analysis}\label{Appendix: accelerated version}
\textbf{A-DEED-GD} is the accelerated version in DEED series. The algorithm and its proof are similar to DEED-GD. The only difference between DEED-GD is the update rule. Please see algorithm \ref{our_nes} below for details.

\begin{algorithm}[h]
\begin{algorithmic}[1]
 \STATE{Initialization: Each server $i\in[N]$ holds $x_0=s_{-1}^i=v_{-1}=0$, server 0 holds $v_{-1}=0$\, $k=0$\;}
 \STATE{Parameter setting: $\eta = \frac{1}{L}$, $\tau = \frac{\sqrt{L}-\sqrt{\mu}}{\sqrt{L}+\sqrt{\mu}}$, $c=\sqrt{1-\sqrt{\frac{\mu}{L}}}$\;}
 \WHILE{the precision is not enough}
  \FOR{$i\in[N]$}
  \STATE{server $i$ computes $g_k^i=\nabla f_i(y_k)$\;}
  \STATE{server $i$ does quantization $d_k^i=Q(g_k^i - s_{k-1}^i, \frac{s {c'}^{k+1}}{2})$\;}
  \STATE{server $i$ updates $s_k^i = d_k^i + s_{k-1}^i$\;}
  \STATE{server $i$ send $d_k^i$ to server 0\;}
  \ENDFOR
  \STATE{server 0 computes $s_k = \frac{1}{N}\sum\limits_{i=1}^N d_k^i + s_{k-1}$\;}
  \STATE{server 0 does quantization $u_k = Q(s_k - v_{k-1}, \frac{s {c'}^{k+1}}{2})$\;}
  \STATE{server 0 sends $u_k$ to server $i,\forall i\in[N]$\;}
  \STATE{server 0 updates $v_k = u_k + v_{k-1}$\;}
  \FOR{$i\in[N]$}
  \STATE{server $i$ updates $v_k = u_k + v_{k-1}$\;}
  \STATE{server $i$ updates $x_{k+1} = y_k - \eta v_k$\;}
  \STATE{server $i$ updates $y_{k+1} = x_{k+1} + \tau(x_{k+1} - x_k)$\;}
  \ENDFOR
 \STATE{$k = k + 1$\;}
 \ENDWHILE
 \end{algorithmic}
 \caption{A-DEED-GD}
\label{our_nes}
\end{algorithm}

\textbf{Proof sketch}
First of all, we can use triangle inequality to prove that the error on $v_k$ is small, i.e. $\|g_k-v_k\|\leq s{c'}^{k+1}$. Then, with diminishing error we are able to prove linear convergence. Finally, we use lemma \ref{quan_vect_lemma} to show that we communicate $O(d)$ bits per communication. The first and the third step are exactly the same as DEED-GD. Hence, we only need to prove that the convergence part.

\begin{proof}
Suppose $f:\mathbb{R}^d\rightarrow\mathbb{R}$ is a $\mu$-convex and $L$-smooth function.

Choose an arbitrary point $x_0=y_0=v_0\in\mathbb{R}^d$, we can define $\phi_0(x):=\phi_0^*+\frac{\mu}{2}\|x-v_0\|^2$ where $\phi_0^*:=f(v_0)$. Then by definition we know $\phi(x)\leq f(x)$ and $\phi_0^*\geq f(x_0)$.

Next, we inductively define the following quantity. Suppose $\ell:\mathbb{Z}^{\geq 0}\rightarrow \mathbb{R}^+$ is an arbitrary function.
\begin{eqnarray}
x_{k+1} &=& y_k - \frac{1}{L} m_k\notag\\
\phi_{k+1}(x)&=&(1-\alpha)\phi_k(x) + \alpha\left[f(y_k) + \left\langle m_k, x - y_k \right\rangle + \frac{\mu}{2}\left\|x-y_k\right\|^2\right]\notag\\
v_{k+1} &=& \arg\min\limits_{v\in\mathbb{R}^n} \phi_{k+1}(v)\notag\\
y_{k+1} &=& \frac{x_{k+1} + \alpha v_{k+1}}{1+\alpha}\notag\\
\phi_{k+1}^* &=& \min \phi_{k+1}\notag
\end{eqnarray}
where $\alpha = \sqrt{\frac{\mu}{L}}$ and $m_k\in\mathbb{R}^d$ such that $\|m_k - \nabla f(y_k)\|\leq c^{k+1}\ell(k)$.

Besides, we will construct monotonically increasing functions $h,g:\mathbb{Z}^{\geq 0}\rightarrow \mathbb{R}$ inductively such that
\begin{eqnarray}
f(x_k) &\leq& \phi_k^* + h(k)\cdot c^{2k}\notag\\
\phi_k(x^*) &\leq& (1-c^{2k})f^* + c^{2k}\phi_0(x^*) + g(k)\cdot c^{2k}.\notag
\end{eqnarray}

Obviously, it is appropriate to set $g(0)=h(0)=0$.

Before we go deeper, here is an important lemma.

\begin{lemma}
With the definition above, we have
\begin{eqnarray}
y_{k+1}=x_{k+1}+\tau(x_{k+1}-x_k)\notag
\end{eqnarray}
where $\tau=\frac{\sqrt{L}-\sqrt{\mu}}{\sqrt{L}+\sqrt{\mu}}$.
\end{lemma}

\begin{proof}
Because $\phi_k(x)=\phi_k^* + \frac{\mu}{2}\|x - v_k\|^2$, taking derivative $\nabla\phi_{k+1}(x) = \mu(1-\alpha)(x-v_k) + \alpha m_k + \alpha\mu(x-y_k)$. And then we get
$v_{k+1}=(1-\alpha)v_k + \alpha y_k - \frac{\alpha}{\mu} m_k$.

Because $y_k=\frac{x_k + \alpha v_k}{1+\alpha}$, we can substitute $v_k = \frac{1+\alpha}{\alpha} y_k - \frac{1}{\alpha}x_k$ and get
\begin{eqnarray}
v_{k+1}&=&\frac{1-\alpha^2}{\alpha}y_k - \frac{1-\alpha}{\alpha}x_k + \alpha y_k - \frac{\alpha}{\mu}m_k\notag\\
&=&\frac{1}{\alpha}\left(y_k - \frac{1}{L}m_k\right) - \frac{1-\alpha}{\alpha}x_k\notag\\
&=&\frac{x_{k+1}-x_k}{\alpha} + x_k.\notag
\end{eqnarray}

Hence
\begin{eqnarray}
y_{k+1}&=&\frac{x_{k+1}+\alpha v_{k+1}}{1+\alpha}\notag\\
&=&x_{k+1}+\tau(x_{k+1}-x_k).\notag
\end{eqnarray}
\end{proof}

\ \\
Then we have
\begin{eqnarray}
f(x_k)-f^*&\leq&  \phi_k^* + h(k)\cdot c^{2k} - f^*\notag\\
&\leq& c^{2k}(\phi_0(x^*)-f^*+g(k)+h(k))\notag\\
&=&c^{2k}(\Delta+g(k)+h(k))\notag
\end{eqnarray}
where $\Delta := \phi_0(x^*)-f^* = f(v_0)-f^*+\frac{\mu}{2}\|x_0-x^*\|^2$. Hence,
\begin{eqnarray}\label{our_acc_convergence}
\|x_k-x^*\|&\leq& \sqrt{\frac{2}{\mu}(f(x_k - x^*))}\notag\\
&=&\sqrt{\frac{2}{\mu}}\cdot c^k\cdot\sqrt{\Delta + g(k) + h(k)}.
\end{eqnarray}

Furthermore, according to lemma 3 and the monotony of $g$ and $h$, we have
\begin{eqnarray}
\|y_k-x^*\|&=&\left\|\frac{2\sqrt{L}}{\sqrt{L}+\sqrt{\mu}}(x_k-x^*)-\frac{\sqrt{L}-\sqrt{\mu}}{\sqrt{L}+\sqrt{\mu}}(x_{k-1}-x^*)\right\|\notag\\
&\leq& \left(\frac{2\sqrt{L}}{\sqrt{L}+\sqrt{\mu}}+\frac{1}{c}\frac{\sqrt{L}-\sqrt{\mu}}{\sqrt{L}+\sqrt{\mu}}\right)\sqrt{\frac{2}{\mu}}\cdot c^k\cdot\sqrt{\Delta + g(k) + h(k)}\notag\\
&\leq& 3\sqrt{\frac{2}{\mu}}\cdot c^k\cdot\sqrt{\Delta + g(k) + h(k)}.\notag
\end{eqnarray}

And then we can give the first upper bound
\begin{eqnarray}
\phi_{k+1}(x^*)&=&(1-\alpha)\phi_k(x^*) + \alpha\left[f(y_k) + \left\langle m_k, x^* - y_k \right\rangle + \frac{\mu}{2}\left\|x^*-y_k\right\|^2\right]\notag\\
&\leq&(1-\alpha)\phi_k(x^*) + \alpha\left[f(y_k) + \left\langle \nabla f(y_k), x^* - y_k \right\rangle + \frac{\mu}{2}\left\|x^*-y_k\right\|^2\right]\notag\\
&&+\alpha\cdot\ell(k)c^{k+1}\|x^*-y_k\|\notag\\
&\leq&(1-\alpha)\phi_k(x^*) + \alpha f^* + \alpha\cdot\ell(k)c^{k+1}\|x^*-y_k\|\notag\\
&\leq& c^2\left( (1-c^{2k})f^* + c^{2k}\phi_0(x^*) + g(k)\cdot c^{2k}\right) + (1-c^2)f^*\notag\\
&&+\alpha\cdot\ell(k)c^{k+1}\|x^*-y_k\|\notag\\
&=& (1-c^{2k+2})f^* + c^{2k+2}\phi_0(x^*) + g(k)\cdot c^{2k+2}\notag\\
&&+3\ell(k)\sqrt{\frac{2}{L}}\cdot c^{2k+1}\cdot\sqrt{\Delta + g(k) + h(k)},\notag
\end{eqnarray}
which means we only need to make
\begin{eqnarray}\label{acc_ineq1}
g(k+1)\geq g(k)+\frac{3\ell(k)\sqrt{\frac{2}{L}}\cdot\sqrt{\Delta + g(k) + h(k)}}{c}
\end{eqnarray}

To make an upper bound of $h(k+1)$, we notice
that
\begin{eqnarray}
\phi_{k+1}^*&=&\phi_{k+1}(v_{k+1})\notag\\
&=&(1-\alpha)\left(\phi_k^* + \frac{\mu}{2}\|v_{k+1}-v_k\|^2\right) + \alpha f(y_k) + \alpha\left\langle m_k, v_{k+1}-y_k\right\rangle\notag\\
&&+\frac{\alpha\mu}{2}\|v_{k+1}-y_k\|^2.\notag
\end{eqnarray}

Substitute $v_{k+1}-y_k=(1-\alpha)(v_k-y_k)-\frac{\alpha}{\mu}m_k$, we have
\begin{eqnarray}
\phi_{k+1}^*&=&(1-\alpha)\phi_k^* + \alpha f(y_k) - \frac{1}{2L}\|m_k\|^2\notag\\
&&+\alpha(1-\alpha)\left(\frac{\mu}{2}\|y_k-v_k\|^2+\left\langle m_k, v_k-y_k\right\rangle\right)\notag\\
&\geq& (1-\alpha)\left(f(x_k)-h(k)\cdot c^{2k}\right) + \alpha f(y_k) - \frac{1}{2L}\|m_k\|^2\notag\\
&&+\alpha(1-\alpha)\left(\frac{\mu}{2}\|y_k-v_k\|^2+\left\langle m_k, v_k-y_k\right\rangle\right).\notag
\end{eqnarray}

Because $f(x_k)\geq f(y_k)+\left\langle \nabla f(y_k), x_k-y_k\right\rangle\geq f(y_k) + \left\langle m_k, x_k-y_k\right\rangle - \ell(k)\cdot c^{k+1}\|x_k-y_k\|$ and $\|x_k-y_k\|=\tau\|x_k-x_{k-1}\|\leq c^2\left(\|x_k-x^*\|+\|x^*-x_{k-1}\|\right)\leq c^{k+1}\cdot(1+c)\sqrt{\frac{2}{\mu}}\sqrt{\Delta+g(k)+h(k)}$.

Hence, we can bound $\phi_{k+1}^*$ by
\begin{eqnarray}
&&f(y_k) - \frac{1}{2L}\|m_k\|^2 - h(k)\cdot c^{2k+2} \notag\\ &&-c^{2k+4}\cdot\ell(k)\cdot (1+c)\sqrt{\frac{2}{\mu}}\sqrt{\Delta+g(k)+h(k)}.\notag
\end{eqnarray}

Recall that
\begin{eqnarray}
f(x_{k+1})-f(y_k)+\frac{1}{2L}\|m_k\|^2&\leq& \frac{1}{L}\left\langle m_k-\nabla f(y_k), m_k\right\rangle\notag\\
&\leq& \frac{1}{L}\ell(k)\cdot c^{k+1}\|m_k\|\notag\\
&\leq& c^{2k+1}\cdot\left(\frac{c\ell^2(k)}{L}+3\sqrt{\frac{2}{\mu}}\ell(k)\sqrt{\Delta+g(k)+h(k)}\right)\notag
\end{eqnarray}
we have
\begin{eqnarray}
\phi_{k+1}^*&\geq& f(x_{k+1})-c^{2k+1}\cdot\left(\frac{c\ell^2(k)}{L}+3\sqrt{\frac{2}{\mu}}\ell(k)\sqrt{\Delta+g(k)+h(k)}\right)\notag\\
&&-h(k)\cdot c^{2k+2} - c^{2k+4}\cdot\ell(k)\cdot (1+c)\sqrt{\frac{2}{\mu}}\sqrt{\Delta+g(k)+h(k)},\notag
\end{eqnarray}
which means $h(\cdot)$ only need to satisfy
\begin{eqnarray}\label{acc_ineq2}
h(k+1)\geq h(k) + \frac{5}{c}\ell(k)\cdot\sqrt{\frac{2}{\mu}}\sqrt{\Delta+g(k)+h(k)} + \frac{\ell^2(k)}{L}.
\end{eqnarray}

Finally, we only need to make inequalities (\ref{acc_ineq1}) and (\ref{acc_ineq2}) be the iteration formulas for arrays $\{g(k)\}_{k\geq0}, \{h(k)\}_{k\geq0}$ with initialization $g(0)=h(0)=0$, and we have
\begin{eqnarray}
i(k+1)\geq i(k) + \left(\frac{3\sqrt{\frac{2}{L}}+5\sqrt{\frac{2}{\mu}}}{c}\right)\ell(k)\sqrt{\Delta+i(k)}+\frac{\ell^2(k)}{L}
\end{eqnarray}
where $i(k):=g(k)+h(k), \forall k\geq0$. Now, $\forall c<c'<1$, define $\gamma = \frac{c'}{c}$ and $\ell(k)=s\gamma^{k+1}$. We will prove that $i(k)\leq C\gamma^{2k},  \forall k\geq0$ for sufficient large $C$ by mathematical induction.

Obviously, it holds for index 0. Suppose it holds for index less or equal to $k$. Then we only need to prove
\begin{eqnarray}
C\gamma^{2k+2}\geq \left(C+\frac{s^2}{L}\right)\gamma^{2k} + \left(\frac{3\sqrt{\frac{2}{L}}+5\sqrt{\frac{2}{\mu}}}{c}\right)s\gamma^{k+1}\sqrt{\Delta+C\gamma^{2k}}.\notag
\end{eqnarray}

This can be derived by
\begin{eqnarray}
C\left(\gamma^2-1\right)\geq\frac{s^2}{L}+\left(\frac{3\sqrt{\frac{2}{L}}+5\sqrt{\frac{2}{\mu}}}{c}\right)s\gamma\sqrt{\Delta+C}.\notag
\end{eqnarray}

Hence, if we define $\alpha_s=\frac{s^2}{L(\gamma^2-1)}+\Delta$, $\beta_s=\left(\frac{3\sqrt{\frac{2}{L}}+5\sqrt{\frac{2}{\mu}}}{c(\gamma^2-1)}\right)s\gamma$, making $C\geq\beta_s^2+\alpha_s+\beta\sqrt{\alpha_s}-\Delta$ is enough. The convergence result follows by inequality (\ref{our_acc_convergence}).
\end{proof}

\section{Experiments}\label{Appendix: experiments}
\subsection{Linear regression}
We illustrate the effectiveness of our proposed algorithms (DEED-GD and its accelerated version A-DEED-GD) in frequent-communication large-memory setting on linear regression problem with a 100 by 100 Gaussian generated matrix with condition number equals to 16. We focus on star networks, and there are 10 computing nodes. We perform 800 epochs on non-accelerated algorithms and 200 epochs on accelerated ones. In each update, the stepsize chosen on each computing node $i$ is $\min_i{1/L_i}$ where $L_i$ is the Lipschitz constant for the function corresponds to node $i$. As required in theorem 3.4 in QSGD's paper\cite{alistarh2017qsgd}, we choose another learning rate for QSGD. The experiments are done on a computer with 2 GHz Dual-Core Intel Core i5 processor. 

We choose the quantization level to be 10000 in QSGD \cite{alistarh2017qsgd} as smaller quantization level lead to larger loss values. For DIANA, as suggested by \cite{mishchenko2019distributed}, we quantize the weight of each layer separately, and we use either the full block size of the weighs or let block size equal to 20. The quantization level in DIANA algorithms is equal to the block size $d_i$ for some vector $i$ and the parameter $\alpha$ in DIANA is chosen to be $\min_i 1 / \sqrt{d_i}$. For DEED-GD, we choose $s = 0.01$ and $c = 0.95$ which is the convergence parameter of the baseline method (GD) using a stepsize $\min_i{1/L_i}$. Here $s, c$ are parameter controlling the maximal error defined below. For A-DEED-GD, we choose $s = 0.1$ and $c = 0.82$ which is the convergence parameter of the accelerated baseline method (A-GD) using a stepsize $\min_i{1/L_i}$.

In our algorithms, to do quantization with maximal error $\varepsilon(=s\cdot c^{k+1})$ at iteration $k$, we consider the following algorithm. For a vector $w$, we first compute a vector $\tilde{w}:=\frac{w\sqrt{d}}{\varepsilon}$. Next, we encode each coordinate of $\tilde{w}$, say $\tilde{w}_i$ into $\left\lfloor \tilde{w}_i \right\rfloor$ and $\left\lfloor \tilde{w}_i \right\rfloor+1$ unbiasedly. We call the new vector $\tilde{v}$. Finally, we compute the quantized vector $v:=\frac{\tilde{v}\varepsilon}{\sqrt{d}}$. It is obvious that using this quantization method the error is bounded by $\varepsilon$.

To encode the integer vector $\tilde{v}$ into an integer, we use the method introduced in QSGD. We refer Elias encoding, a map from positive integer to non-negative integer. First of all, We use Elias to encode the first non-zero element of $\tilde{v}$ (use 1 bit to encode sign and use other bits to encode absolute value) and its position. Then we literately use Elias to encode the next non-zero element and its distance from the previous position. It works perfectly especially when $\tilde{v}$ is sparse.

As mentioned in table \ref{table_large}, we noticed that there could be two settings for algorithm only does quantization on computing nodes like DIANA. In star network setting, the center node transmits full vector (with no quantization) to computing nodes in broadcasting term. This would cost extra $32dN$ bits. In fully connected network setting, computing nodes broadcast to each other and update information separately. This would lead to an extra $N-1$ factor on total number of bits since each computing node should broadcast to $N-1$ other nodes. One another method is to time 2 instead of time $N-1$ on total number of bits from computing nodes to center node to compare them with DEED series. It is reasonable for the following reasons. 1) In DEED series, we have two communications for each computing nodes in each iteration: sending message and receiving message. So, in this case what we really compare is 
proportional to $\text{number of iterations}\times\text{number of bits per communication}$. 2) The ``$\times 2$'' scheme counts less bits than both star network setting and fully connected network setting. If we can beat DIANA and other algorithms in this setting, it means our framework is essentially better than their framework, i.e. better than them in any settings. Notice that this is only a method of counting bits, not a method of communication.

The performance analysis has already provided in Section \ref{sec: experiments} and it shows that our algorithms (DEED-GD and its accelerated version A-DEED-GD) save the most number of bits in communication without scarifying the convergence speed. We also run the experiments on 5 different random seeds and the results is shown in Figure \ref{fig:linear regression multiple runs}. In Figure \ref{fig:linear regression multiple runs}, the shaded regions line up with the maximal and minimal loss values at each epoch among the 5 different runs, and the regions are too small to visualize due to small variance.

\begin{figure}
\centering
\includegraphics[width=80mm]{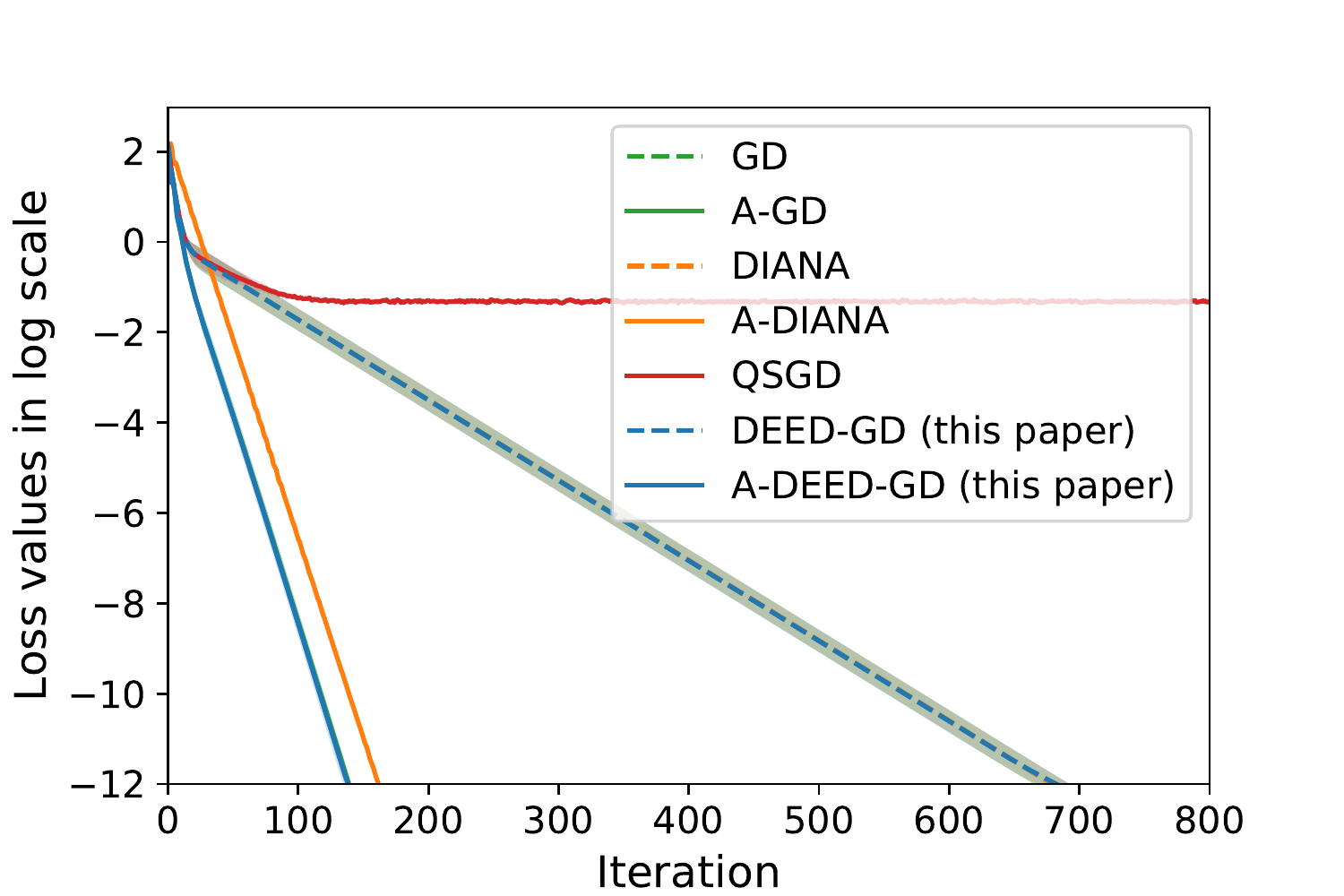}
\caption{Compare loss values of each algorithms among 5 different runs}
\label{fig:linear regression multiple runs}
\end{figure}

\subsection{Image classification on MNIST}

We evaluate the effectiveness of the proposed algorithms via
training a neural network on the MNIST
dataset \cite{lecun1998gradient} for image classification. 
MNIST consists of 60,000 $28 \times 28$ pixel training images containing a single numerical digit and an additional 10,000 test examples.
Our neural network consists of one 500-neuron fully-connected layer followed by a ten unit softmax layer for classification, and the layer used reLU activations \cite{nair2010rectified}. 
The experiments are performed on a using NVIDIA GeForce GTX1080
GPU, and the models are distributed over 6 computing servers, where each of the servers have access to 10,000 training images. 
In large-memory setting, each server uses its own 10,000 images to update the models, while in small-memory setting, each server uses randomly selected 1666 images among its own 10,000 training images as a minibatch. We train the models for 200 epochs in large-memory setting and 100 epochs for small-memory setting. The training time is approximately 9 hours for each algorithm in small-memory setting and 6 hours for large-memory setting.

We compare our algorithms with QSGD \cite{alistarh2017qsgd}, DIANA \cite{mishchenko2019distributed}, DoubleSqueeze \cite{tang2019doublesqueeze}, and Terngrad \cite{wen2017terngrad} in frequent communication settings.
For all the algorithms except DIANA, in every update, we vectorize the gradient or the gradient difference of the weight matrices of the neural network, concatenate these vectors as a large vector and do quantization, and reshape the quantized vector into the original shapes of the weight matrices for updating. 
For DIANA, as suggested by \cite{mishchenko2019distributed}, we quantize the weight of each layer separately, and we use two different block sizes in quantization (use the full block size of the weighs, or let block size equal to 128). As suggested by \cite{mishchenko2019distributed}, the quantization level in DIANA algorithms is equal to the block size $d_i$ for some vector $i$ and the parameter $\alpha$ in DIANA is chosen to be $\min_i 1 / \sqrt{d_i}$.
In large-memory setting, we train our algorithm DEED-GD with the parameters $e = 0.1, s = 25$, and in small-memory setting, we train DEED-SGD with parameters $e = 0.2, s = 25$. Here the maximal error $\varepsilon:=\frac{s}{(k+1)^e}$. We use the same encoding algorithm as described in linear regression.
In both large-memory and small-memory setting, we present the results for choosing 4-bit quantization for QSGD.
For DoubleSqueeze, we perform two kinds of quantization as discussed in \cite{tang2019doublesqueeze}: top-k compression and 1-bit compression.
For fair comparison, we use the same stepsize for all the algorithms, where the stepsize for large-memory setting is 0.25 and is 1.18 for small-memory setting. These stepsizes are chosen as the loss curves of the baseline methods (GD and SGD) are smooth and the baseline methods achieve fairly high accuracy in testing. As shown in Figure \ref{fig:MNIST loss GD} and Figure \ref{fig:MNIST loss SGD}, our algorithms as well as QSGD, DIANA and Terngrad achieve the same loss values as the baseline methods (GD and SGD) in both large-memory setting and small-memory setting.

To compare the efficiency of each algorithm, we compute the total number of bits throughout the training. 
As large integers are less frequent in encoded vectors \cite{alistarh2017qsgd}, we use Elias integer encoding to save bits in communication \cite{elias1975universal} for all algorithms in comparison. 
Notice that Terngrad, DIANA and QSGD only perform quantization on computing nodes and will typically use 32-bit precision to encode the vectors which are sent from the center node in a star network. 
To be fair in comparison, we use the scaling technique proposed in \cite{wen2017terngrad} and use $\log_2(1+2*N) * d$ bits to encode 
the vectors sent from center node where $N$ is number of computing servers and $d$ is the dimension of the vector. 
This number is significantly smaller than $32 * d$ unless $N \geq 2^{30}$. 
For DIANA and QSGD, we let computing nodes to share information to each other so they do not need to broadcast via the center node which will cause $32 *d$ bits for each update. 
Under this setting, the bits communicated in each update is $B * (N-1)$ where $B$ is the number of bits to communicated from the computing node and $N-1$ is the number of other computing nodes that need to communicate to. 
In most cases, the bits computed in this way is fewer than using 32-bit precision to encode the vectors which are broadcast from the center node. 

\textbf{Results of Efficiency.} To illustrate the efficiency of our proposed algorithms, we plot the number of bits vs the testing accuracy in Figure \ref{fig:MNIST GD} and \ref{fig:MNIST SGD} for frequent large-memory setting and frequent small-memory setting. 
In both figures, the curve that is located on the left-most corresponds to our proposed algorithms, DEED-GD and DEED-SGD. 
This means that our proposed algorithms use the fewest total number of bits to achieve the accuracy. 

To better illustrate how many bits we can save from other algorithms, we present the total number of bits in Table \ref{Table: MNIST GD} and Table \ref{Table: MNIST SGD} and compute the ratio between the number of bits required by other algorithms and the number of bits required by the proposed algorithms in the fourth column of the Table. 
For example, in Table \ref{fig:MNIST GD}, the number $10.44$ means DEED-GD takes 10.44 times fewer the number of bits than QSGD to achieve the accuracy. In theory, DIANA and our proposed algorithm DEED-GD both have linear convergence rate, but our experiments show that we can take 190.28 times fewer the number of bits than DIANA (block size equals to 128) to achieve the similar performance in training. In addition, our algorithms achieve the highest testing accuracy in the final epoch as shown in the second column of the table, and the 
accuracy is comparable or even higher than the ones achieved by the non-quantized baseline algorithms (GD achieves $91.7\%$ and SGD achieves $97.37\%$).

To ensure the performance are reproducible, we also train the models under different random seeds and choose different parameters within certain ranges. For example, we train the models by our proposed algorithm using the parameters $e$ is chosen from $[0.1, 0.3]$, $s$ is chosen from $\{16, 25, 32\}$. The comparison we discussed above is still valid under these changes. We also run experiments for different quantization levels (e.g. using 2-bit quantization or 3-bit quantization) for QSGD and DIANA, but they cannot achieve the same testing accuracy with the same number of epochs as using 4-bit quantization, so we do not discuss these results here.

\begin{figure}
\centering     
\subfigure[Large-Memory Setting]
{\label{fig:MNIST loss GD}\includegraphics[width=65mm]{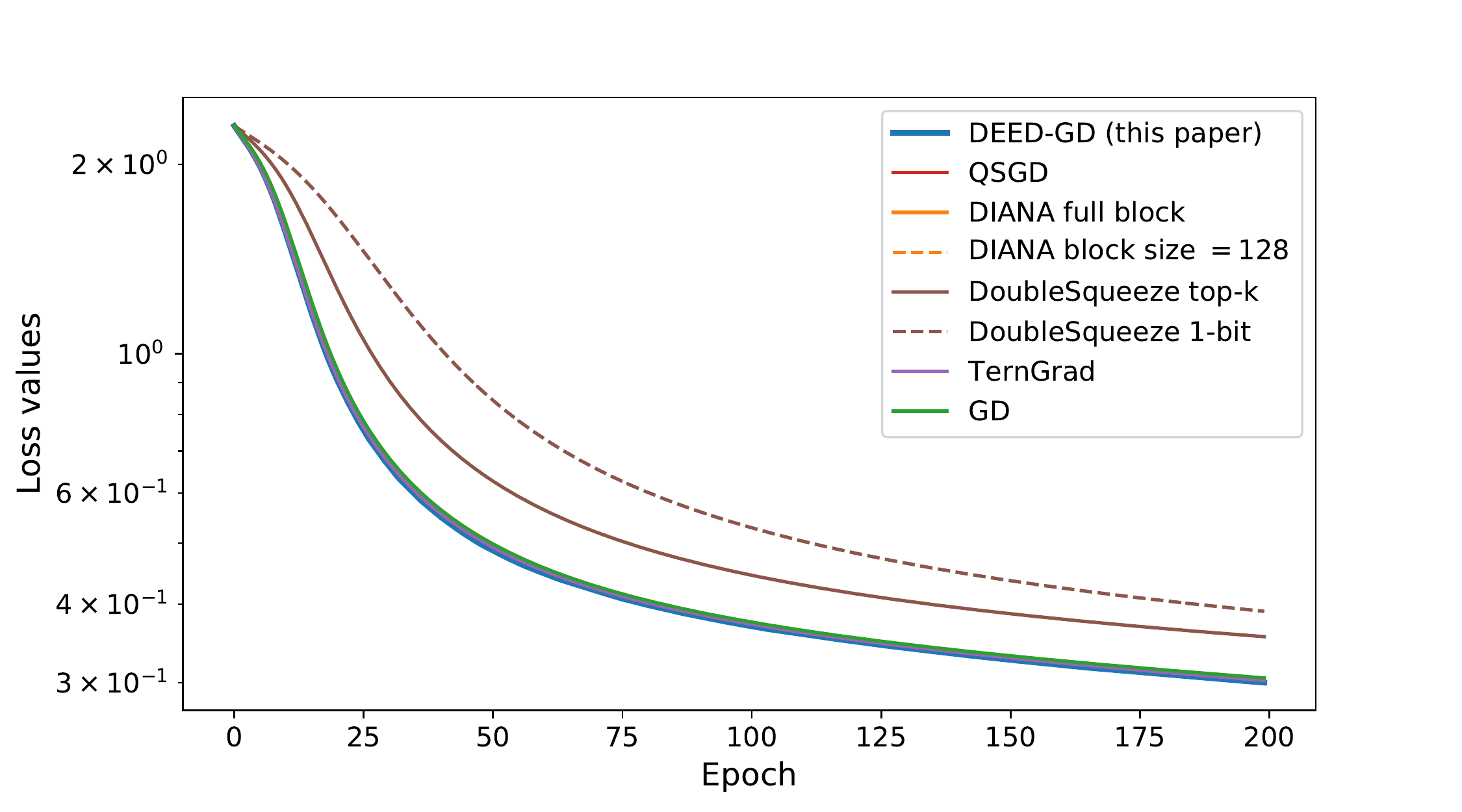}}
\subfigure[Small-Memory Setting]
{\label{fig:MNIST loss SGD}\includegraphics[width=65mm]{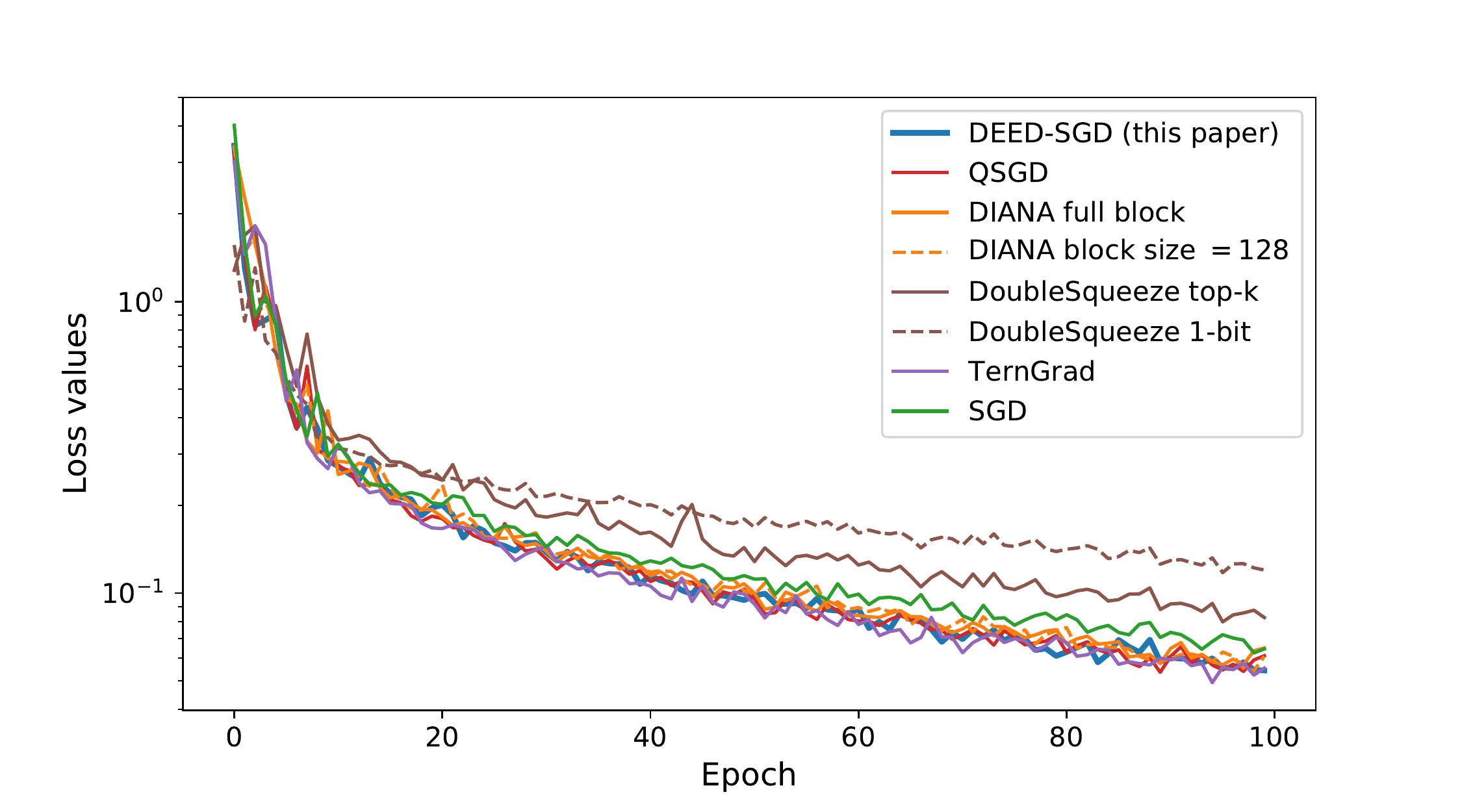}}
\caption{Compare loss values of each algorithms}
\end{figure}

\begin{figure}
\centering     
\subfigure[Large-Memory Setting]
{\label{fig:MNIST GD}\includegraphics[width=65mm]{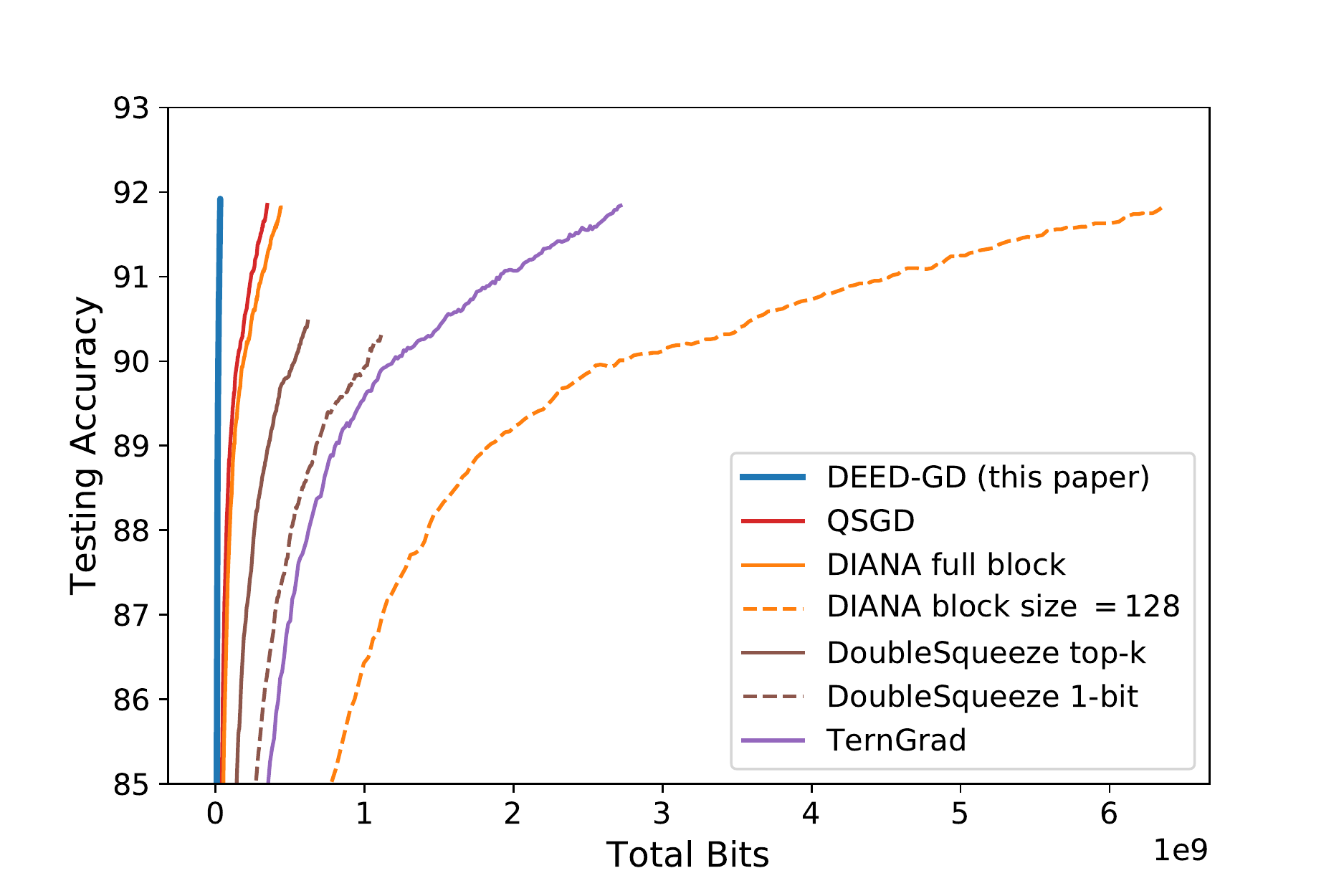}}
\subfigure[Small-Memory Setting]
{\label{fig:MNIST SGD}\includegraphics[width=65mm]{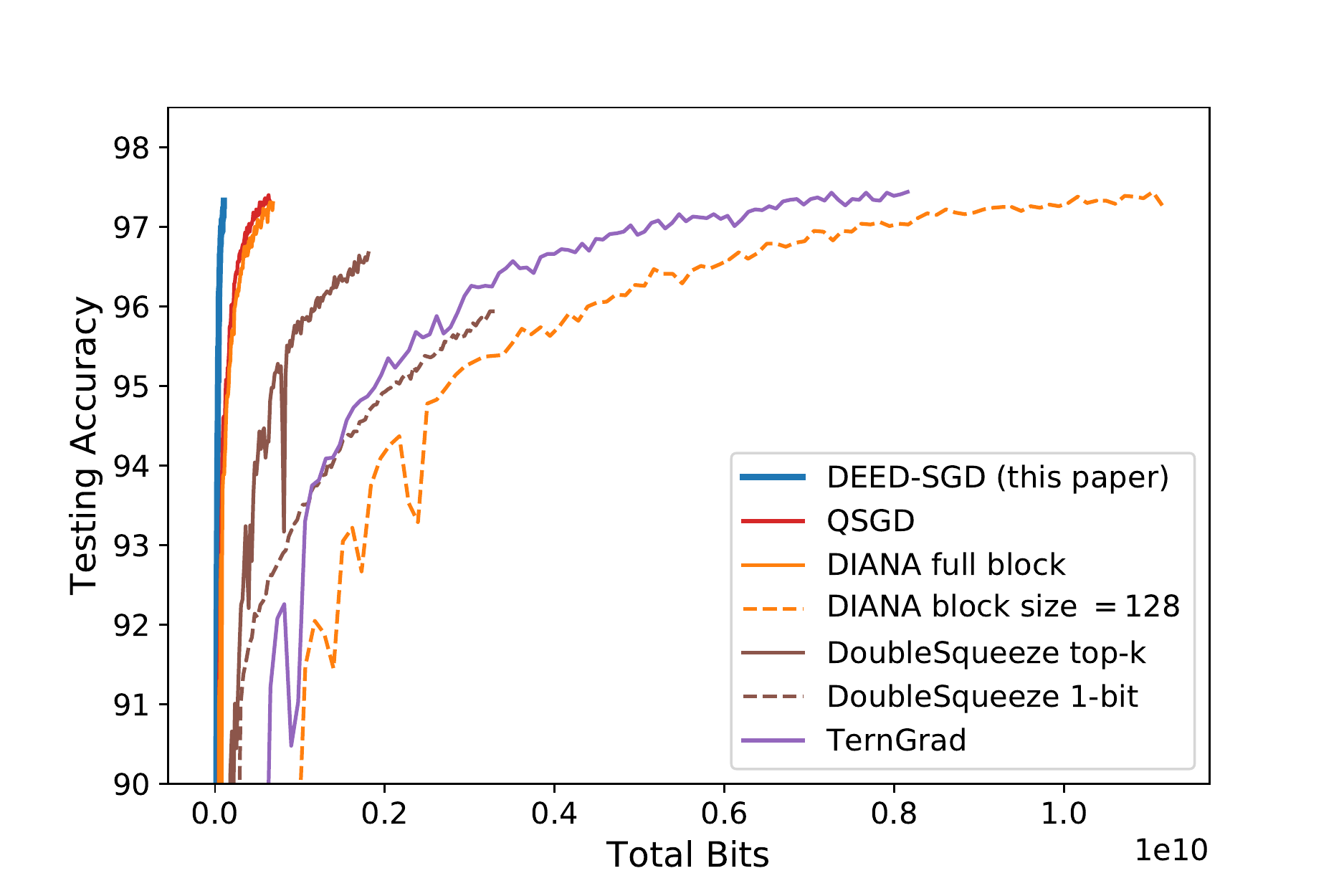}}
\caption{Total number of bits to reach certain testing accuracy}
\end{figure}

\begin{table}[]
\centering
\renewcommand\arraystretch{1.18}
\begin{tabular}{lccc}
\toprule
Algorithm & Testing accuracy & Total number of bits & Ratio  \\ 
\midrule
\textbf{DEED-GD}  & 91.86\% & \num{3.34d07} & 1.00\\
\midrule
\textbf{QSGD} \cite{alistarh2017qsgd} & 91.85\% & \num{3.49d08} & 10.44\\
\textbf{DIANA (full block)} \cite{mishchenko2019distributed} & 91.82\% & \num{4.40d08} & 13.17\\
\textbf{DIANA (block size $= 128$)} \cite{mishchenko2019distributed} & 91.82\% & \num{6.35d09} & 190.28\\
\textbf{DoubleSqueeze top-k} \cite{tang2019doublesqueeze} & 90.47\% & \num{6.21d08} & 18.59\\
\textbf{DoubleSqueeze 1-bit} \cite{tang2019doublesqueeze} & 90.31\% & \num{1.11d09} & 33.34\\
\textbf{TernGrad} \cite{wen2017terngrad} & 91.84\% & \num{2.72d09} & 81.44\\
\bottomrule
\end{tabular}
\vspace{2mm}
\caption{DEED-GD saves bits in communication for large-memory setting}
\label{Table: MNIST GD}
\end{table}

\begin{table}[]
\centering
\renewcommand\arraystretch{1.18}
\begin{tabular}{lccc}
\toprule
Algorithm & Testing accuracy & Total number of bits & Ratio  \\ 
\midrule
\textbf{DEED-SGD}  & 97.33\% & \num{1.04d+08} &  1.00 \\
\midrule
\textbf{QSGD} \cite{alistarh2017qsgd} & 97.31\% & \num{6.52d+08} &  6.25 \\
\textbf{DIANA (full block)} \cite{mishchenko2019distributed} & 97.30\% & \num{6.83d+08} &  6.55 \\
\textbf{DIANA (block size $= 128$)} \cite{mishchenko2019distributed} & 97.27\% & \num{1.12d+10} &  106.95 \\
\textbf{DoubleSqueeze top-k} \cite{tang2019doublesqueeze} & 96.67\% & \num{1.81d+09} &  17.31 \\
\textbf{DoubleSqueeze 1-bit} \cite{tang2019doublesqueeze} & 95.96\% & \num{3.34d+09} &  32.01 \\
\textbf{TernGrad} \cite{wen2017terngrad} & 97.44\% & \num{8.16d+09} &  78.20 \\
\bottomrule
\end{tabular}
\vspace{2mm}
\caption{DEED-SGD saves bits in communication for small-memory setting}
\label{Table: MNIST SGD}
\end{table}

\newpage
\bibliographystyle{ieeetr}

\bibliography{main}

\end{document}